\documentclass[12pt]{amsart}

\usepackage{graphics, amsmath, amsfonts, amssymb, amsthm, amscd, mathrsfs}

\newtheorem{theorem}{Theorem}

\newtheorem{lemma}[theorem]{Lemma}

\theoremstyle{definition}
\newtheorem{definition}[theorem]{Definition}

\newtheorem{remark}[theorem]{Remark}

\title{Continuing the classification of homogeneous Kobayashi-hyperbolic manifolds with high-dimensional automorphism group}

\author{Elliot Herrington}

\address{School of Mathematical Sciences, University of Adelaide, Adelaide SA 5005, Australia}
\email{elliot.herrington@adelaide.edu.au}



\begin{document} 

\begin{abstract}  
\vspace{1.5cm} 
We classify all homogeneous Kobayashi-hyperbolic manifolds of dimension $n \ge 2$ whose group of holomorphic automorphisms has dimension either $n^2 - 7$ or $n^2 - 8.$ This paper continues the work of A. Isaev, who classified all such manifolds with automorphism group dimension $n^2 - 6$ and greater. 
\end{abstract}

\maketitle

\section{Introduction} 

A Kobayashi-hyperbolic manifold is a connected complex manifold on which the Kobayashi pseudodistance is a true distance. Such manifolds have been extensively studied since their introduction in the late 1960s and have many interesting properties. The class of Kobayashi-hyperbolic manifolds is quite large, and includes all bounded domains and many unbounded domains in $\mathbb C^n$. In the literature, Kobayashi-hyperbolic manifolds are usually referred to as simply \emph{hyperbolic}, and we follow this convention from here on. 

Let $M$ denote an $n$-dimensional hyperbolic manifold, and set $d(M):= \dim \text{Aut} (M).$ It is well-known that the maximal dimension of $d(M)$ is given by $d(M) = n^2 + 2n$, which holds if and only if $M$ is biholomorphic to the unit ball in $\mathbb C^n.$ In a series of papers from 2005 to 2008, Isaev classified all hyperbolic manifolds with automorphism group dimension $n^2 - 1 \leq d(M) \leq n^2 + 2n$ (see the monograph \cite{Isa4} for a consolidation of these results). Whilst it would be desirable to extend the classification beyond the critical value of $d(M) = n^2 -1,$ providing a full explicit description of all hyperbolic manifolds of automorphism group dimension $d(M) = n^2 -2$ is impossible. For instance, a generic Reinhardt domain in $\mathbb C^2$ has a $2$-dimensional automorphism group. Such domains have uncountably many isomorphism classes, so cannot be explicitly described. 

Further progress is possible if we restrict the class of manifolds under consideration, taking $M$ to be \emph{homogeneous}. That is, $M$ is a complex hyperbolic manifold on which Aut($M$) acts transitively. Such manifolds are of general interest in complex geometry, and Isaev was able to continue the classification by restricting attention to this class. In particular, all homogeneous hyperbolic manifolds with automorphism group dimension $n^2-6 \leq d(M) \leq n^2-2$ were classified (see papers \cite{Isa2}, \cite{Isa1}, \cite{Isa3}). In this article, we continue the classification by 'working downwards,' and provide a description of all homogeneous hyperbolic manifolds satisfying $d(M) = n^2 - 7$ or $d(M) = n^2 - 8.$ We see that, up to biholomorphism, there are two such manifolds with automorphism group dimension $d(M) = n^2 - 7,$ and four such manifolds with automorphism group dimension $d(M)= n^2 - 8.$ In what follows, $B^n$ will denote the $n$-dimensional open unit ball. 

\begin{theorem} \label{main7}
Let $M$ be a homogeneous $n$-dimensional Kobayashi-hyperbolic manifold with $d(M) = n^2-7.$ Then one of the following holds: 
\begin{enumerate}
\item $n=5$ and $M$ is biholomorphic to $B^2 \times T_3$, where $T_3$ is the tube domain 
\begin{align}
T_3 = \big\{(z_1, z_2, z_3) \in \mathbb C^3 : (\operatorname{Im }z_1)^2 -(\operatorname{Im }z_2)^2 - (\operatorname{Im }z_3)^2 > 0, \operatorname{Im }z_1 >0 \big\}. \nonumber 
\end{align} 
\item $n = 5$ and $M$ is biholomorphic to $B^1 \times T_4$, where $T_4$ is the tube domain
\begin{align}
T_4 = \big\{(z_1, z_2, z_3, z_4) \in \mathbb C^4 : (\operatorname{Im }z_1)^2 -(\operatorname{Im }&z_2)^2 - (\operatorname{Im }z_3)^2 - (\operatorname{Im }z_4)^2>0, \nonumber \\ 
&\operatorname{Im }z_1 >0 \big\}. \nonumber 
\end{align} 
\end{enumerate} 
\end{theorem} 

\begin{theorem} \label{main8}
Let $M$ be a homogeneous $n$-dimensional Kobayashi-hyperbolic manifold with $d(M) = n^2-8.$ Then one of the following holds: 
\begin{enumerate}
\item $n=5$ and $M$ is biholomorphic to $B^1 \times B^1 \times B^1 \times B^2$.
\item $n=6$ and $M$ is biholomorphic to the tube domain 
\begin{align}
T_6 = \big\{(z_1, z_2, z_3, z_4, z_5, z_6&) \in \mathbb C^6 : (\operatorname{Im } z_1)^2 -(\operatorname{Im } z_2)^2 - (\operatorname{Im } z_3)^2 \nonumber \\
&- (\operatorname{Im } z_4)^2 - (\operatorname{Im } z_5)^2 - (\operatorname{Im } z_6)^2 >0, \operatorname{Im }z_1 >0 \big\}. \nonumber 
\end{align} 
\item $n=7$ and $M$ is biholomorphic to $B^1 \times B^1 \times B^5$. 
\item $n=8$ and $M$ is biholomorphic to $B^2 \times B^6$. 
\end{enumerate} 
\end{theorem} 

We now present the classification itself. Combined with the classical fact for dimension $d(M)=n^2+2n$, the results collected in \cite{Isa4}, and the articles \cite{Isa2}, \cite{Isa1} and \cite{Isa3}, the above two theorems yield the following classification for homogeneous hyperbolic manifolds with $d(M) \ge n^2-8$ up to biholomorphism: 
\vspace{3mm} 

\noindent {\bf Classification.} Let $M$ be an homogeneous $n$-dimensional Kobayashi-hyperbolic manifold satisfying $n^2-8\le d(M)\le n^2+2n$. Then $M$ is biholomorphic either to a product of unit balls, a tube domain, a product of a unit ball and a tube domain, or to the domain $\mathcal D$ given below. Specifically, $M$ is one of the following manifolds: 
\begin{itemize}
\item[{\rm (i)}] $B^n$ {\rm (}here $d(M)=n^2+2n${\rm )}.
\item[{\rm (ii)}] $B^1\times B^{n-1}$ {\rm (}here $d(M)=n^2+2${\rm )}.
\item[{\rm (iii)}] $B^1\times B^1\times B^1$ {\rm (}here $n=3$, $d(M)=9=n^2${\rm )}.
\item[{\rm (iv)}] $B^2\times B^2$ {\rm (}here $n=4$, $d(M)=16=n^2${\rm )}.
\item[{\rm (v)}] $B^1\times B^1\times B^2$ {\rm (}here $n=4$, $d(M)=14=n^2-2${\rm )}.
\item[{\rm (vi)}] $B^2\times B^3$ {\rm (}here $n=5$, $d(M)=23=n^2-2${\rm )}.
\item[{\rm (vii)}] $B^1\times B^1\times B^1\times B^1$ {\rm (}here $n=4$, $d(M)=12=n^2-4${\rm )}.
\item[{\rm (viii)}] $B^1\times B^1\times B^3$ {\rm (}here $n=5$, $d(M)=21=n^2-4${\rm )}.
\item[{\rm (ix)}] $B^2\times B^4$ {\rm (}here $n=6$, $d(M)=32=n^2-4${\rm )}.
\item[{\rm (x)}] $B^1\times B^2\times B^2$ {\rm (}here $n=5$, $d(M)=19=n^2-6${\rm )}.
\item[{\rm (xi)}] $B^3\times B^3$ {\rm (}here $n=6$, $d(M)=30=n^2-6${\rm )}.
\item[{\rm (xii)}] $B^1\times B^1\times B^4$ {\rm (}here $n=6$, $d(M)=30=n^2-6${\rm )}.
\item[{\rm (xiii)}] $B^2\times B^5$ {\rm (}here $n=7$, $d(M)=43=n^2-6${\rm )}.
\item[{\rm (xiv)}] $B^1\times B^1\times B^1\times B^2$ {\rm (}here $n=5$, $d(M)=17=n^2-8${\rm )}.
\item[{\rm (xv)}] $B^1\times B^1\times B^5$ {\rm (}here $n=7$, $d(M)=41=n^2-8${\rm )}.
\item[{\rm (xvi)}] $B^2\times B^6$ {\rm (}here $n=8$, $d(M)=56=n^2-8${\rm )}.
\item[{\rm (xvii)}] the tube domain $T_3$ defined in Theorem 1 {\rm (}here $n=3$, $d(M)=10=n^2+1${\rm )}.
\item[{\rm (xviii)}] the tube domain $T_4$ defined in Theorem 1 {\rm (}here $n=4$, $d(M)=15=n^2-1${\rm )}.
\item[{\rm (xix)}] the tube domain $T_5$ given by 
\begin{align}
T_5 = \big\{(z_1, z_2, z_3, z_4, z_5) \in \mathbb C^5 &: (\operatorname{Im } z_1)^2 -(\operatorname{Im } z_2)^2 - (\operatorname{Im } z_3)^2 \nonumber \\
&- (\operatorname{Im } z_4)^2 - (\operatorname{Im } z_5)^2 >0, \operatorname{Im }z_1 >0 \big\} \nonumber 
\end{align} 
{\rm (}here $n=5$, $d(M)=21=n^2-4${\rm )}. 
\item[{\rm (xx)}] the tube domain $T_6$ defined in Theorem 2 {\rm (}here $n=6$, $d(M)=28=n^2-8${\rm )}.
\item[{\rm (xxi)}] $B^1\times T_3$ {\rm (}here $n=4$, $d(M)=13=n^2-3${\rm )}.
\item[{\rm (xxii)}] $B^2\times T_3$ {\rm (}here $n=5$, $d(M)=18=n^2-7${\rm )}.
\item[{\rm (xxiii)}] $B^1\times T_4$ {\rm (}here $n=5$, $d(M)=18=n^2-7${\rm )}.
 \item[{\rm (xxiv)}] the domain $\mathcal D$ given by 
\begin{align}
\mathcal D = \big\{ (z, w) \in \mathbb C^3 \times \mathbb C : (\operatorname{Im } z_1 &- |w|^2)^2 - (\operatorname{Im } z_2 - |w|^2)^2 - (\operatorname{Im } z_3) > 0, \nonumber \\
& \operatorname{Im } z_1 -|w|^2 > 0 \big\} \nonumber 
\end{align} 
{\rm (}here $n=4$, $d(M)=10=n^2-6${\rm )}. 
\end{itemize}

Note that $\mathcal D$, the final domain listed in the above classification, is linearly equivalent to the well-known example of a non-symmetric bounded homogeneous domain in $\mathbb C^4$, discovered by I. Pyatetskii-Shapiro (see \cite[pp.~26-28]{PS}). The proofs of Theorems 1 and 2 rely on an important structure theorem due to Nakajima (see \cite{N}), which states that every homogeneous hyperbolic manifold is biholomorphic to an affinely homogeneous Siegel domain of the second kind (to be defined in the next section). After utilising this theorem, we then proceed to analyse the (graded) Lie algebra of the automorphism group of a Siegel domain of the second kind. An explicit description of this Lie algebra, which is quite involved, was provided in \cite{KMO} and \cite{S}. We begin the following section by defining a Siegel domain of the second kind, before considering its automorphism group and associated Lie algebra. 
\vspace{3mm} 

\noindent {\bf Acknowledgements.} This paper represents part of the PhD thesis of the author under the supervision of Finnur L\'arusson and Thomas Leistner, and the author would like to thank them both for their supervision and guidance. He would also like to thank the late Alexander Isaev, under whose supervision this project began. The author gratefully acknowledges the support of an Australian Government Research Training Program (RTP) scholarship. 

\section{Siegel domains of the second kind and other preliminaries} 

Before providing the definition of a Siegel domain of the second kind, it is necessary to consider the concept of an open convex cone, and its automorphism group. An open subset $\Omega \subset \mathbb R^k$ is called an \emph{open convex cone} if and only if $x, y \in \Omega$ implies $\lambda x + \mu y \in \Omega$ for all $\lambda, \mu >0.$ That is, $\Omega$ is a convex cone if it is closed with respect to taking linear combinations of its elements with positive coefficients. 

\begin{definition} \label{autdef}
The \emph{linear automorphism group} of an open convex cone $\Omega$ is defined by 
$$G(\Omega) = \left\{ A \in \operatorname{GL}_k (\mathbb R) : A \Omega = \Omega \right\}.$$ 
\end{definition} 

We will hereafter refer to $G(\Omega)$ as simply the \emph{automorphism group.} Clearly, $G(\Omega)$ is a closed subgroup of $\operatorname{GL}_k(\mathbb R)$, and hence is a Lie group. We denote by $\mathfrak g (\Omega) \subset \mathfrak{gl} (\mathbb R)$ its Lie algebra. The open convex cone $\Omega$ is said to be $homogeneous$ if $G(\Omega)$ acts on $\Omega$ transitively, that is, if for all $x, y \in \Omega$ there exists $A \in G(\Omega)$ such that $Ax = y.$ 

We will be concerned with the notion of a \emph{proper} open convex cone, that is, an open convex cone that does not contain a line through the origin.\footnote{We follow the usage in \cite{FK}. Unfortunately, the terminology is not standard. For example, in \cite{BV} the term is reserved for convex cones that are closed, have non-empty interior, and do not contain a line through the origin.} For such cones, we have a useful estimate for the dimension of $\mathfrak g(\Omega).$ The following proof is due to Isaev (see \cite[Lemma 2.1]{Isa1}). For completeness, we present it here in full. 

\begin{lemma} \label{gest} 
If $\Omega \subset \mathbb R^k$ is a proper open convex cone, then 
$$\dim G(\Omega) \leq \frac{k^2}{2} - \frac{k}{2} +1.$$ 
\end{lemma} 

\begin{proof} 
For a fixed  $x \in \Omega$, consider the isotropy subgroup $G(\Omega)_x \subset G(\Omega).$ This subgroup is compact, since it leaves invariant the bounded open set $\Omega \cap (x - \Omega).$ Therefore, changing variables in $\mathbb R^k$ if necessary, we can assume that $G(\Omega)_x$ lies in the orthogonal group $O_k(\mathbb R).$ The group $O_k(\mathbb R)$ acts transitively on the sphere of radius $||x||$ in $\mathbb R^k,$ and the isotropy subgroup $I_x$ of $x$ under the $O_k(\mathbb R)$-action is isomorphic to $O_{k-1}(\mathbb R).$ Since $G(\Omega)_x \subset I_x,$ we have 
$$\dim G(\Omega)_x \leq \dim I_x = \frac{(k-2)(k-1)}{2} = \frac{k^2}{2} - \frac{3k}{2} +1,$$ 
from which the result follows. 
\end{proof} 

Now, let $H: \mathbb C^m \times \mathbb C^m \to \mathbb C^k$ be a Hermitian form on $\mathbb C^m$ taking values in $\mathbb C^k,$ where we follow the convention of linearity in the second variable and conjugate linearity in the first variable. Then $H$ is called $\Omega$-\emph{Hermitian} if for an open convex cone $\Omega \subset \mathbb R^k$ we have $H(w, w) \in \bar{\Omega} \setminus \left\{ 0 \right\}$ for all non-zero $w \in \mathbb C^m.$ Observe that if $H$ is $\Omega$-Hermitian and $\Omega$ is proper, then there exists a positive-definite linear combination of the components of $H.$ 

\begin{definition} \label{Siegeldomain}
A \emph{Siegel domain of the second kind} in $\mathbb C^n$ is a domain of the form
$$
S(\Omega,H):=\left\{(z,w)\in\mathbb C^k\times\mathbb C^{n-k}: \operatorname{Im } z-H(w,w)\in \Omega\right\}
$$
for some $1\le k\le n$, some open convex cone $\Omega\subset\mathbb R^k$, and some $\Omega$-Hermitian form $H$ on $\mathbb C^{n-k}$. 
\end{definition} 

Note that for $k=n$ we have $H=0$, so in this case $S(\Omega,H)$ is linearly equivalent to the domain  
$$
\left\{z\in\mathbb C^n:\operatorname{Im } z\in \Omega\right\}.
$$
At the other extreme, when $k=1$, the domain $S(\Omega,H)$ is linearly equivalent to
$$
\left\{(z,w)\in\mathbb C\times\mathbb C^{n-1}:\operatorname{Im } z-||w||^2>0\right\},
$$ 
which is an unbounded realization of the unit ball $B^n$ (see \cite[p.~31]{R}). 
In fact, the above result holds in more generality. If 
$$\Omega = \left\{x \in \mathbb R^k: x_1 > 0, \ldots, x_k >0 \right\}$$ 
and $S(\Omega, H)$ is homogeneous, then $S(\Omega, H)$ is linearly equivalent to a product of $k$ unbounded realisations of unit balls as above, and hence biholomorphic to a product of unit balls (see \cite[Theorems~A, B, C]{KT} and \cite{N1}). We now describe the group of holomorphic affine automorphisms of a Siegel domain of the second kind (see \cite[pp.~25--26]{PS}).

\begin{theorem}\label{Siegelaffautom}
Any holomorphic affine automorphism of $S(\Omega,H)$ has the form
$$
(z, w) \mapsto (Az+a+2iH(b,Bw)+iH(b,b), Bw+b), 
$$
with $a\in\mathbb R^k$, $b\in\mathbb C^{n-k}$, $A\in G(\Omega)$, $B\in \operatorname{GL}_{n-k}(\mathbb C)$, where
\begin{equation}
AH(w,w')=H(Bw,Bw')\label{assoc}
\end{equation}
for all $w,w'\in\mathbb C^{n-k}$.
\end{theorem}

A domain $S(\Omega,H)$ is called \emph{affinely homogeneous} if the group described in the above theorem, which we denote $\operatorname{Aff}(S(\Omega,H))$, acts on $S(\Omega,H)$ transitively. Denote by $G(\Omega,H)$ the subgroup of $G(\Omega)$ that consists of all transformations $A\in G(\Omega)$ as in Theorem \ref{Siegelaffautom}, that is, of all elements $A\in G(\Omega)$ for which there exists $B\in\ \operatorname{GL}_{n-k}(\mathbb C)$ such that (\ref{assoc}) holds. By \cite[Lemma 1.1]{D}, the subgroup $G(\Omega,H)$ is closed in $G(\Omega)$. It can be deduced from Theorem \ref{Siegelaffautom} that if $S(\Omega,H)$ is affinely homogeneous, the action of $G(\Omega,H)$ is transitive on $\Omega$ (see, e.g., \cite[proof of Theorem 8]{KMO}), so the cone $\Omega$ is homogeneous. Note also that since $\Omega$ is connected, the action of the identity component $G(\Omega, H)^\circ$ is also transitive on $\Omega.$ Conversely, if $G(\Omega,H)$ acts on $\Omega$ transitively, the domain $S(\Omega,H)$ is affinely homogeneous. 

After realising a homogeneous Kobayashi-hyperbolic manifold as a homogeneous Siegel domain of the second kind using the main result of \cite{N}, we then consider its automorphism group $\operatorname{Aut}(S(\Omega,H))$, and proceed by analysing the Lie algebra of this group, which we denote ${\mathfrak g}(S(\Omega,H))$. We therefore rely heavily on an explicit description of this Lie algebra (see \cite[Theorems 4 and 5]{KMO}), which is rather involved, and we present it now. This algebra is isomorphic to the (real) Lie algebra of complete holomorphic vector fields on $S(\Omega,H)$ (see \cite[pp.~209--210]{S}). 

\begin{theorem}\label{kmoalgebradescr}
The algebra ${\mathfrak g}={\mathfrak g}(S(\Omega,H))$ admits a grading
$$
{\mathfrak g}={\mathfrak g}_{-1}\oplus{\mathfrak g}_{-1/2}\oplus{\mathfrak g}_0\oplus{\mathfrak g}_{1/2}\oplus{\mathfrak g}_1,
$$
with ${\mathfrak g}_{\nu}$ being the eigenspace with eigenvalue $\nu$ of $\operatorname{ad }\partial$, where
$$\partial:=z\cdot\frac{\partial}{\partial z}+\frac{1}{2}w\cdot\frac{\partial}{\partial w}.$$ 
Here
$$
\begin{array}{ll}
{\mathfrak g}_{-1}=\displaystyle\left\{a\cdot\frac{\partial}{\partial z}:a\in\mathbb R^k\right\},&\dim {\mathfrak g}_{-1}=k,\\
\vspace{-0.1cm}\\
{\mathfrak g}_{-1/2}=\displaystyle\left\{2i H(b,w)\cdot\frac{\partial}{\partial z}+b\cdot\frac{\partial}{\partial w}:b\in\mathbb C^{n-k}\right\},&\dim {\mathfrak g}_{-1/2}=2(n-k),
\end{array}
$$
and ${\mathfrak g}_0$ consists of all vector fields of the form
\begin{equation}
(Az)\cdot\frac{\partial}{\partial z}+(Bw)\cdot\frac{\partial}{\partial w},\label{g0}
\end{equation}
with $A\in{\mathfrak g}(\Omega)$, $B\in{\mathfrak{gl}}_{n-k}(\mathbb C)$ and
\begin{equation}
AH(w,w')=H(Bw,w')+H(w,Bw')\label{assoc1}
\end{equation}
for all $w,w'\in\mathbb C^{n-k}$. Furthermore, we have 
\begin{equation} 
\dim {\mathfrak g}_{1/2}\le 2(n-k),\qquad \dim {\mathfrak g}_1\le k.\label{estimm}
\end{equation}
\end{theorem}

The subspace $\mathfrak g_0$ is in fact a subalgebra of $\mathfrak g$. The matrices $A$ that appear in (\ref{g0}) clearly form the Lie algebra of $G(\Omega,H),$ and ${\mathfrak g}_{-1}\oplus{\mathfrak g}_{-1/2}\oplus{\mathfrak g}_0$ is isomorphic to the Lie algebra of the group $\operatorname{Aff }(S(\Omega,H))$. This can be seen by comparing conditions (\ref{assoc}) and (\ref{assoc1}). 

Following the approach in \cite{S}, for a pair of matrices $A,B$ satisfying (\ref{assoc1}) we say that $B$ is \emph{associated to $A$} (with respect to $H$). Let ${\mathcal L}$ be the (real) subspace of ${\mathfrak{gl}}_{n-k}(\mathbb C)$ of all matrices associated to the zero matrix in ${\mathfrak g}(\Omega)$, i.e., matrices skew-Hermitian with respect to each component of $H$. That is, 
$$\mathcal L = \left\{ B \in \mathfrak{gl}_{n-k}(\mathbb C) : H(B \cdot, \cdot) + H(\cdot, B \cdot) = 0 \right\}.$$ 
We now set $s:=\dim {\mathcal L}$, and note that 
\begin{equation}
\dim {\mathfrak g}_0 = s + \dim G(\Omega, H) \leq s+\dim {\mathfrak g}(\Omega).\label{estim1}
\end{equation}
By Theorem \ref{kmoalgebradescr} and the above inequality we obtain
\begin{equation}
d(S(\Omega,H))\le k+2(n-k)+s+\dim {\mathfrak g}(\Omega)+\dim {\mathfrak g}_{1/2}+ \dim {\mathfrak g}_1,\label{estim 8}
\end{equation}
which, combined with (\ref{estimm}), leads to
\begin{equation}
d(S(\Omega,H))\le 2k+4(n-k)+s+\dim {\mathfrak g}(\Omega).\label{estim2}
\end{equation}
The subspace $\mathcal L$ lies in the Lie algebra of matrices skew-Hermitian with respect to any linear combination $\mathbf H$ of the components of the Hermitian form $H$. Since $\mathbf H$ can be chosen to be positive-definite, 
\begin{equation}
s\le (n-k)^2.\label{ests}
\end{equation}
The above inequality combined with (\ref{estim2}) yields 
\begin{equation}
d(S(\Omega,H))\le 2k+4(n-k)+(n-k)^2+\dim {\mathfrak g}(\Omega).\label{estim3}
\end{equation}
Lastly, combining (\ref{estim3}) with inequality (\ref{gest}) we deduce the following useful upper bound: 
\begin{equation}
d(S(\Omega,H))\le\displaystyle\frac{3k^2}{2}-k\left(2n+\frac{5}{2}\right)+n^2+4n+1.\label{estim4}
\end{equation}

In Theorem \ref{kmoalgebradescr} above, explicit descriptions of the first three components of the graded Lie algebra are provided. Explicit descriptions of $\mathfrak g_{1/2}$ and $\mathfrak g_1$ are considerably more complicated, and were first given in \cite{S} (see Chapter V, Propositions 2.1 and 2.2). We present these now, beginning with the  $\mathfrak g_{1/2}$ component. 

\begin{theorem}\label{descrg1/2}
The subspace ${\mathfrak g}_{1/2}$ consists of all vector fields of the form
$$
2iH(\Phi(\bar z),w)\cdot\frac{\partial}{\partial z}+(\Phi(z)+c(w,w))\cdot\frac{\partial}{\partial w},
$$
where $\Phi:\mathbb C^k\to\mathbb C^{n-k}$ is a $\mathbb C$-linear map such that for every ${\mathbf w}\in\mathbb C^{n-k}$ one has
\begin{equation}
\Phi_{{\mathbf w}}:=\left[x\mapsto \operatorname{Im } H({\mathbf w},\Phi(x)),\,\, x\in\mathbb R^k\right]\in{\mathfrak g}(\Omega),\label{Phiw0}
\end{equation}
and $c:\mathbb C^{n-k}\times\mathbb C^{n-k}\to\mathbb C^{n-k}$ is a symmetric $\mathbb C$-bilinear form on $\mathbb C^{n-k}$ with values in $\mathbb C^{n-k}$ satisfying the condition
\begin{equation}
H(w,c(w',w'))=2iH(\Phi(H(w',w)),w')\label{cond1}
\end{equation}
for all $w,w'\in\mathbb C^{n-k}$. 
\end{theorem}

Further, the component ${\mathfrak g}_1$ admits the following description.

\begin{theorem}\label{descrg1}
The subspace ${\mathfrak g}_1$ consists of all vector fields of the form
$$
a(z,z)\cdot\frac{\partial}{\partial z}+b(z,w)\cdot\frac{\partial}{\partial w},
$$
where $a:\mathbb R^k\times\mathbb R^k\to\mathbb R^k$ is a symmetric $\mathbb R$-bilinear form on $\mathbb R^k$ with values in $\mathbb R^k$ {\rm (}which we extend to a symmetric $\mathbb C$-bilinear form on $\mathbb C^k$ with values in $\mathbb C^k${\rm )} such that for every ${\mathbf x}\in\mathbb R^k$ one has
\begin{equation}
A_{{\mathbf x}}:=\left[x\mapsto a({\mathbf x},x),\,\,x\in\mathbb R^k\right]\in{\mathfrak g}(\Omega),\label{idents1}
\end{equation}
and $b:\mathbb C^k\times\mathbb C^{n-k}\to\mathbb C^{n-k}$ is a $\mathbb C$-bilinear map such that, if for ${\mathbf x}\in\mathbb R^k$ one sets
\begin{equation}
B_{{\mathbf x}}:=\left[w\mapsto\frac{1}{2}b({\mathbf x},w),\,\,w\in\mathbb C^{n-k}\right],\label{idents2}
\end{equation}
the following conditions are satisfied:
\begin{itemize}
\item[{\rm (i)}] $B_{{\mathbf x}}$ is associated to $A_{{\mathbf x}}$ and $\operatorname{Im} \operatorname{tr} B_{{\mathbf x}}=0$ for all ${\mathbf x}\in\mathbb R^k$,
\item[{\rm (ii)}] for every pair ${\mathbf w},{\mathbf w}'\in\mathbb C^{n-k}$ one has
$$
B_{{\mathbf w},{\mathbf w}'}:=\left[x\mapsto\Im H({\mathbf w'},b(x,{\mathbf w})),\,\,x\in\mathbb R^k\right]\in{\mathfrak g}(\Omega),
$$
\item[{\rm (iii)}] $H(w,b(H(w',w''),w''))=H(b(H(w'',w),w'),w'')$ for all $w,w',w''\in\mathbb C^{n-k}$.
\end{itemize} 
\end{theorem} 

We now recall the well-known classification, up to linear equivalence, of homogeneous open convex cones in dimensions $k=2,3,4,5$ not containing entire lines (see \cite[pp.~38--41]{KT}). Utilising the notation in \cite[p.~7]{FK}, $\Lambda_n$ denotes the $n$-dimensional Lorentz cone. 
\begin{itemize}

\item [$k=2$:] 
\begin{itemize}
\item[]
$\Omega_1:=\left\{(x_1,x_2)\in\mathbb R^2:x_1>0,\,\,x_2>0\right\}$, where the algebra ${\mathfrak g}(\Omega_1)$ consists of all diagonal matrices, hence $\dim {\mathfrak g}(\Omega_1)=2$,
\end{itemize}
\vspace{0.3cm}

\item [$k=3$:] 
\begin{itemize}
\item[(i)] $\Omega_2:=\left\{(x_1,x_2,x_3)\in\mathbb R^3:x_1>0,\,\,x_2>0,\,\,x_3>0\right\}$, where the algebra ${\mathfrak g}(\Omega_2)$ consists of all diagonal matrices, hence $\dim {\mathfrak g}(\Omega_2)=3$,
\vspace{0.1cm}
\item[(ii)] $\Omega_3:= \Lambda_3 = \left\{(x_1,x_2,x_3)\in\mathbb R^3:x_1^2-x_2^2-x_3^2>0,\,\,x_1>0\right\}$, where we have ${\mathfrak g}(\Omega_3)={\mathfrak c}({\mathfrak{gl}}_3(\mathbb R))\oplus{\mathfrak{so}}_{1,2}$, hence $\dim {\mathfrak g}(\Omega_3)=4$; here for any Lie algebra ${\mathfrak h}$ we denote by ${\mathfrak c}({\mathfrak h})$ its centre,
\end{itemize}
\vspace{0.3cm}

\item [$k=4$:] 
\begin{itemize}
\item[(i)] $\Omega_4:=\left\{(x_1,x_2,x_3,x_4)\in\mathbb R^4:x_1>0,\,\,x_2>0,\,\,x_3>0,\,\,x_4>0\right\}$, 
\newline where the algebra ${\mathfrak g}(\Omega_4)$ consists of all diagonal matrices, hence we have $\dim {\mathfrak g}(\Omega_4)=4$,
\vspace{0.1cm}
\item[(ii)] $\Omega_5:= \Lambda_3 \times \mathbb R_+ = \big\{(x_1,x_2,x_3,x_4)\in\mathbb R^4: x_1^2-x_2^2-x_3^2>0,\,\,x_1>0, \\ 
\hspace*{7cm}x_4>0\big\}$, 
\newline where the algebra ${\mathfrak g}(\Omega_5)=\left({\mathfrak c}({\mathfrak{gl}}_3(\mathbb R))\oplus{\mathfrak{so}}_{1,2}\right)\oplus\mathbb R$ consists of block-diagonal matrices with blocks of sizes $3\times 3$ and $1\times 1$ corresponding to the two summands, hence $\dim {\mathfrak g}(\Omega_5)=5$,
\vspace{0.1cm}
\item[(iii)] $\Omega_6:= \Lambda_4 = \left\{(x_1,x_2,x_3,x_4)\in\mathbb R^4: x_1^2-x_2^2-x_3^2-x_4^2>0,\,\,x_1>0\right\}$, 
\newline where ${\mathfrak g}(\Omega_6)={\mathfrak c}({\mathfrak{gl}}_4(\mathbb R))\oplus{\mathfrak{so}}_{1,3}$, hence $\dim {\mathfrak g}(\Omega_6)=7$,
\end{itemize}
\vspace{0.3cm} 

\item [$k=5$:] 
\begin{itemize}
\item[(i)] $\Omega_7:=\big\{(x_1,x_2,x_3,x_4, x_5)\in\mathbb R^5:x_1>0,\,\,x_2>0,\,\,x_3>0,\,\,x_4>0, \\
\hspace*{7cm}x_5>0\big\}$, 
\newline where the algebra ${\mathfrak g}(\Omega_7)$ consists of all diagonal matrices, hence we have $\dim {\mathfrak g}(\Omega_7)=5$,
\vspace{0.1cm}
\item[(ii)] $\Omega_{8}:= \Lambda_3 \times \mathbb R^2_+ = \big\{(x_1,x_2,x_3,x_4, x_5)\in\mathbb R^5: x_1^2-x_2^2-x_3^2>0,\,\,x_1>0, \\
\hspace*{7cm} x_4>0,\,\,x_5>0\big\}$, 
\newline where the algebra ${\mathfrak g}(\Omega_{8})=\left({\mathfrak c}({\mathfrak{gl}}_3(\mathbb R))\oplus{\mathfrak{so}}_{1,2}\right)\oplus\mathbb R \oplus \mathbb R$ consists of block-diagonal matrices with blocks of sizes $3\times 3$, $1\times 1$ and $1 \times 1$ corresponding to the three summands, hence $\dim {\mathfrak g}(\Omega_{8})=6$,
\vspace{0.1cm}
\item[(iii)] $\Omega_{9}:= \Lambda_4 \times \mathbb R_+ = \big\{(x_1,x_2,x_3,x_4, x_5)\in\mathbb R^5: x_1^2-x_2^2-x_3^2-x_4^2>0,\,\,x_1>0,\\
\hspace*{7cm} x_5>0\big\}$, 
\newline where the algebra ${\mathfrak g}(\Omega_{9})=\left({\mathfrak c}({\mathfrak{gl}}_3(\mathbb R))\oplus{\mathfrak{so}}_{1,3}\right)\oplus\mathbb R$ consists of block-diagonal matrices with blocks of sizes $4\times 4$ and $1\times 1$ corresponding to the two summands, hence $\dim {\mathfrak g}(\Omega_{9})=8$,
\vspace{0.1cm}
\item[(iv)] $\Omega_{10}:= \Lambda_5 = \big\{(x_1,x_2,x_3,x_4, x_5)\in\mathbb R^5: x_1^2-x_2^2-x_3^2-x_4^2-x_5^2>0, \\
\hspace*{7cm} x_1>0\big\}$,
\newline where ${\mathfrak g}(\Omega_{10})={\mathfrak c}({\mathfrak{gl}}_4(\mathbb R))\oplus{\mathfrak{so}}_{1,4}$, hence $\dim {\mathfrak g}(\Omega_{10})=11$,
\vspace{0.1cm} 
\item[(v)] 
$\Omega_{11} := \big\{ (x_1, x_2, x_3, x_4, x_5) \in \mathbb R^5 : x_1 > 0,\,\,x_1x_2- x_4^2 > 0, \\
\hspace*{7cm} x_1x_2x_3 - x_3x_4^2 - x_2x_5^2 > 0 \big\},$ \\ 
where $\dim \mathfrak g(\Omega_{11}) = 5$, which is proved below,
\vspace{0.1cm} 
\item[(vi)] 
$\Omega_{12} := \big\{ (x_1, x_2, x_3, x_4, x_5) \in \mathbb R^5 : x_1 > 0,\,\,x_1x_2 - x_4^2 > 0, \\
\hspace*{7cm} x_1x_3 - x_5^2 >0 \big\},$ 
\newline where $\dim \mathfrak g(\Omega_{12}) = 5$, which is proved below. 
\end{itemize}

\end{itemize}

We conclude this section by providing a justification of the above claim that $\dim \mathfrak g(\Omega_{11}) = \dim \mathfrak g(\Omega_{12}) = 5$. In the case of the first ten homogeneous open convex cones in the list, each cone is either a positive orthant in some dimension, a Lorentz cone, or a product of such cones. In this situation, the Lie algebra of the automorphism group of each cone is straightforward to compute. In contrast, the cones $\Omega_{11}$ and $\Omega_{12}$ are considerably more complicated, and the task of merely determining their automorphism groups is rather involved. 

To begin with, the homogeneous open convex cone $\Omega_{11}$ is referred to in the literature as the \emph{dual Vinberg cone}, and the cone $\Omega_{12}$ as the \emph{Vinberg cone}. These two cones were first introduced in \cite{V2}, and as their names suggest are in fact dual to each other (see e.g. \cite[Example 4]{GI} for a proof that the cones are dual, or \cite[Theorem 3.3.1]{H} for an elementary proof). The automorphism group of a proper open convex cone is isomorphic to the automorphism group of its dual (see \cite[Proposition I.1.7]{FK}), and so the dimensions of their automorphism groups are equal. We now provide a brief sketch of how the automorphism group of the dual Vinberg cone $G(\Omega_{11})$ is determined (for details, see \cite[pp.~126--128]{IK}). 

Let $S^n(\mathbb R)$ denote the vector space of real symmetric $n \times n$ matrices, and let $S^n_+$ denote the subset of real positive definite symmetric $n \times n$ matrices. Consider the subspace $V$ of $S^3(\mathbb R)$, given by 
\[
V = \left\{ X= \left[\begin{array}{ccc} x_1 & x_4 & x_5 \\ x_4 & x_2 & 0 \\ x_5 & 0 & x_3  \end{array} \right] : x_1, x_2, x_3, x_4, x_5 \in \mathbb R \right\}, 
\]
and also the subset of $V$ given by those matrices in $V$ that are positive-definite, which we denote $\mathscr C$. That is, $\mathscr C = V \cap S^3_+ \subset S^3_+ \subset S^3(\mathbb R)$. 
For $X \in \mathscr C$, since positive definiteness of a symmetric matrix is equivalent to the positivity of the determinant of each principal submatrix, we see that the positive definiteness of $X$ necessitates that $x_1 > 0, x_1x_2-x_4^2 > 0,$ and $x_1x_2x_3 - x_3x_4^2 - x_2x_5^2 > 0$, and these conditions exactly describe the cone $\Omega_{11}$. So we see that the cone $\Omega_{11}$ is mapped bijectively onto $\mathscr{C}$ by the function $f: \Omega_{11} \to \mathscr C$, which is given by 
$$f(x_1, x_2, x_3, x_4, x_5) = \left[\begin{array}{ccc} x_1 & x_4 & x_5 \\ x_4 & x_2 & 0 \\ x_5 & 0 & x_3  \end{array} \right].$$  

Following \cite{IK}, we now define the Lie group 
\[ H := \left\{ A = \left[ \begin{array}{ccc} a & b & c \\ 0 & e & 0 \\ 0 & 0 & i \end{array} \right] : a, b, c, e, i \in \mathbb R \text{ with } a >0, ei \neq 0 \right\} \] 
and let $H^+$ be the subgroup consisting of matrices in $H$ with positive diagonal entries. Then $H^+$ is the connected identity component of $H$. Let $\rho: H \to \operatorname{GL}(V)$ be the representation of $H$ given by 
\[ \rho(A)X := AXA^T \] 
where $A \in H$ and $X \in V.$ It is straightforward to show that $\rho$ is a faithful representation. Further, $H$ and $H^+$ act transitively on the cone $\mathscr C \subset V$ by $\rho.$ By an intricate argument given in \cite[pp.~126--128]{IK}, the automorphism group of the cone $\mathscr C$, $G(\mathscr C),$ is given by 
\[ G(\mathscr C) = \rho(H^+) \rtimes G(\mathscr C)_{I_3} \] 
where $G(\mathscr C)_{I_3}$ is the isotropy subgroup of $I_3 \in \mathscr C$, which is finite. Since $H^+$ is five-dimensional, we see that $G(\mathscr C)$ is five-dimensional, and it is this result that is utilised in our classification. 

\section{Proof of Theorem 1} 

Let $M$ be a homogeneous Kobayashi-hyperbolic manifold of dimension $n$. By \cite{N}, the manifold $M$ is biholomorphic to an affinely homogeneous Siegel domain of the second kind $S(\Omega, H)$. Recall that 
$$
S(\Omega,H):=\left\{(z,w)\in\mathbb C^k\times\mathbb C^{n-k}: \operatorname{Im } z-H(w,w)\in \Omega\right\},
$$
for $1 \leq k \leq n$, where $\Omega \subset \mathbb R^k$ is an open convex cone and $H$ is an $\Omega$-Hermitian form on $\mathbb C^{n-k}$.
Since all homogeneous Kobayashi-hyperbolic manifolds of dimensions 2 and 3 have been classified (see \cite[Theorem~2.6]{Isa1}), we take $n \geq 4$. Further, we recall from the remarks after Definition \ref{Siegeldomain} in the previous section that if $k=1$ then $S(\Omega, H)$ is biholomorphic to $B^n$, so we assume that $k \geq 2.$ We can use the following lemma to rule out a large number of remaining possibilities. 

\begin{lemma} 
For $n \geq 6$ and $k \geq 4$, we cannot have $d(S(\Omega, H)) = n^2 - 7$. Also, for $n \geq 8$ and $k = 3$, we cannot have $d(S(\Omega, H)) = n^2 - 7$. 
\end{lemma} 

\begin{proof} 
We will show that for $n \geq 6, k \geq 4$, as well as for $n \geq 8, k = 3$, the right-hand side of inequality (\ref{estim4}) given by 
$$ d(S(\Omega, H)) \leq \frac{3k^2}{2} - \left(2n + \frac{5}{2} \right) k + n^2 +4n + 1 $$
is strictly less that $n^2 - 7$. That is, for these $k, n$ the following holds: 
$$ \frac{3k^2}{2} - \left(2n + \frac{5}{2} \right) k  + 4n + 8 < 0. $$
To see this, consider the quadratic function 
$$ \varphi(t) := \frac{3t^2}{2} - \left(2n + \frac{5}{2} \right) t + 4n + 8. $$ 
The discriminant of $\varphi$ is given by 
$$ \mathcal{D} := 4n^2 - 14n - \frac{167}{4}, $$ 
which is positive for $n \geq 6$. The zeroes of $\varphi$ are given by 
$$t_1 := \frac{2n + \frac{5}{2} - \sqrt{\mathcal D}}{3},$$
$$t_2 := \frac{2n + \frac{5}{2} + \sqrt{\mathcal D}}{3}.$$ 

To prove the lemma, it suffices to show that: (i) $t_2 > n$ for $n \geq 6$, (ii) $t_1 < 4$ for $n \geq 6$, and (iii) $t_1 < 3$ for $n \geq 8$. Beginning with the inequality $t_2 > n$, we have 
$$n - \frac{5}{2} < \sqrt{\mathcal D},$$
or, equivalently, that 
$$n^2 - 3n - 16 > 0,$$ 
which clearly holds for $n \geq 6$. Now considering $t_1 < 4$, we see that 
$$2n - \frac{19}{2} < \sqrt{\mathcal D}, $$ 
or, equivalently, that 
$$n > \frac{11}{2}, $$ 
which holds for $n \geq 6.$ Lastly, the inequality $t_1 < 3$ implies that 
$$2n - \frac{13}{2} < \sqrt{\mathcal D}, $$
or, equivalently, that 
$$n > 7, $$ 
which holds for $n \geq 8$. This completes the proof.  
\end{proof} 

By the above lemma, we prove the theorem by considering the following eight cases: 
\begin{enumerate}
\item $k = 2, n \geq 4$ 
\item $k = 3, n = 4$
\item $k = 3, n = 5$ 
\item $k = 3, n = 6$ 
\item $k = 3, n = 7$
\item $k = 4, n = 4$ 
\item $k = 4, n = 5$ 
\item $k = 5, n = 5.$
\end{enumerate}
We now begin by considering each case.

\textbf{Case 1.} Suppose that $k = 2, n \geq 4.$  Since $H: \mathbb C^{n-k} \times \mathbb C^{n-k} \to \mathbb C^k$, we have that $H = (H_1, H_2)$ is a pair of Hermitian forms on $\mathbb C^{n-2}$. After a linear change of $z$-variables, we may assume that $H_1$ is positive-definite. Since this is the case, by applying a linear change of $w$-variables, we can simultaneously diagonalise $H_1$ and $H_2$ as 
$$H_1(w,w) = ||w||^2, \hspace{5mm} H_2(w,w) = \sum_{j=1}^{n-2} \lambda_j |w_j|^2. $$ 
If all the eigenvalues of $H_2$ are equal, $S(\Omega, H)$ is linearly equivalent either to 
$$D_1 := \left\{ (z, w) \in \mathbb C^2 \times \mathbb C^{n-2} : \text{Im } z_1 - ||w||^2 > 0, \text{ Im } z_2 > 0 \right\}$$ 
if $\lambda_j = 0,$ or to 
 $$D_2 := \left\{ (z, w) \in \mathbb C^2 \times \mathbb C^{n-2} : \text{Im } z_1 - ||w||^2 > 0, \text{ Im } z_2 - ||w||^2 > 0 \right\}$$ 
if $\lambda_j \neq 0.$ The domain $D_1$ is biholomorphic to $B^{n-1} \times B^1$, hence $d(D_1) = n^2 + 2 > n^2 - 7,$ which shows that $S(\Omega, H)$ cannot be equivalent to $D_1$. As for the domain $D_2$, consider the group $G(\Omega_1, (||w||^2, ||w||^2))$. A straightforward computation shows that $G(\Omega_1, (||w||^2, ||w||^2))$ consists of matrices $\left[\begin{array}{cc} \rho & 0 \\ 0 & \rho \end{array} \right]$ where $\rho >0$, and $\left[\begin{array}{cc} 0 & \eta \\ \eta & 0 \end{array} \right]$ where $\eta >0.$ It is therefore seen that the action of $G(\Omega_1, (||w||^2, ||w||^2))$ is not transitive on $\Omega_1$. Therefore, $S(\Omega, H)$ cannot be equivalent to $D_2$ either. It follows that $H_2$ has at least one pair of distinct eigenvalues. 

Next, since $\dim \mathfrak{g}(\Omega) = 2$, inequality (\ref{estim2}) yields 
\begin{equation} 
s \geq n^2 - 4n - 5. \label{eq1} 
\end{equation} 
On the other hand, by inequality (\ref{ests}),
\begin{equation} 
s \leq n^2 - 4n + 4.  \label{eq2} 
\end{equation} 
The exact value of $s$ is given by  
$$s = n^2 - 4n + 4 - 2m, $$
where $m \geq 1$ is the number of pairs of distinct eigenvalues of $H_2$. This fact is a consequence of the following lemma.

\begin{lemma}\label{dimsubspaceskewherm} \it Let ${\mathcal H}$ be a Hermitian matrix of size $r\times r$ and ${\mathcal K}$ the real vector space of skew-Hermitian matrices of size $r\times r$ that are at the same time skew-Hermitian with respect to ${\mathcal H}$:
\[
{\mathcal K}:=\left\{B\in{\mathfrak{gl}}_r(\mathbb C): B + B^*=0,\,\,{\mathcal H}B+B^*{\mathcal H}=0\right\}.
\]
Then $\dim{\mathcal K}=r^2-2p$, where $p$ is the number of unordered pairs of distinct eigenvalues of ${\mathcal H}$, counted with multiplicity. Hence, if $\dim{\mathcal K}=r^2$, then ${\mathcal H}$ is a scalar matrix. 
\end{lemma}
\begin{proof} 
Note first that $\mathcal K$ is the centraliser of $\mathcal H$ in $\mathfrak u(r)$. That is, 
\[ \mathcal K = \left\{ B \in \mathfrak u(r) : \mathcal H B - B \mathcal H = 0 \right\}. \] 
Since $B \in \mathcal K$ commutes with $\mathcal H$, it preserves each eigenspace of $\mathcal H$, so 
\[ \mathcal K \cong \mathfrak u(r_1) \oplus \cdots \oplus \mathfrak u(r_k), \] 
where $r_1, \ldots, r_k$ are the dimensions of the eigenspaces of $\mathcal H$, and $k$ is the number of distinct eigenvalues of $\mathcal H.$ Hence, 
\begin{equation} 
\begin{split} 
\dim \mathcal K &= r_1^2 + \cdots + r_k^2 \\ 
&= (r_1 + \cdots + r_k)^2 - 2 \sum_{i < j} r_ir_j \\ 
&= r^2-2p, \nonumber 
\end{split} 
\end{equation} 
which completes the proof. 
\end{proof}

By (\ref{eq1}) and (\ref{eq2}) above, we see that we must have $1 \leq m \leq 4$, which leads to the following possibilities: 
\begin{enumerate}
\item $n = 4$ and $\lambda_1 \neq \lambda_2$ (here $m=1$ and $s=2$), 
\item $n=5$ and, upon permutation of $w$-variables, $\lambda_1 \neq \lambda_2 = \lambda_3$ (here $m=2$ and $s=5$), 
\item $n=5$ and $\lambda_1, \lambda_2, \lambda_3$ are pairwise distinct (here $m=3$ and $s=3$), 
\item $n=6$ and, upon permutation of $w$-variables, $\lambda_1 \neq \lambda_2 = \lambda_3 = \lambda_4$ (here $m=3$ and $s=10$), 
\item $n=6$ and, upon permutation of $w$-variables, $\lambda_1 = \lambda_2 \neq \lambda_3 = \lambda_4$ (here $m=4$ and $s=8$), or
\item $n=7$ and, upon permutation of $w$-variables, $\lambda_1 \neq \lambda_2 = \lambda_3 = \lambda_4 = \lambda_5$ (here $m=4$ and $s=17$). 
\end{enumerate}

We know from the discussion after Definition \ref{Siegeldomain} in the previous section that in this case (when $k = 2, n \geq 4$) $S(\Omega, H)$ is biholomorphic to a product of two unit balls $B^l \times B^{n-l}$ for $1 \leq l \leq n-1$. The dimension of its automorphism group is given by 
$$d(B^l \times B^{n-l}) = 2l^2 - 2nl + n^2 + 2n.$$ 
Since $n$ is limited to the range $4,5,6,7,$ we set the right-hand side equal to $n^2 - 7$ and solve for $l$ in each case. For none of the above values of $n$ is $l$ integer-valued, and therefore this case makes no contributions to our classification. 
\vspace{2mm} 

\textbf{Case 2.} 
Suppose that $k = 3, n = 4$. Then $S(\Omega, H)$ is linearly equivalent to either 
$$D_3 := \left\{ (z,w) \in \mathbb C^3 \times \mathbb C : \text{Im } z - v|w|^2 \in \Omega_2 \right\},$$ 
where $v = (v_1, v_2, v_3)$ is a vector in $\mathbb R^3$ with non-negative entries, or 
$$D_4 := \left\{ (z,w) \in \mathbb C^3 \times \mathbb C : \text{Im } z - v|w|^2 \in \Omega_3 \right\},$$ 
where $v = (v_1, v_2, v_3)$ is a vector in $\mathbb R^3$ satisfying $v_1^2 \geq v_2^2 + v_3^2,$ $v_1 > 0.$ Let us consider each of these cases separately. 

Assume that $S(\Omega, H)$ is equivalent to the domain $D_3$. Since $\Omega_2$ is equivalent to the positive orthant in $\mathbb R^3$ then $S(\Omega, H)$ must be biholomorphic to a four-dimensional product of three unit balls, and it is immediate to see that the only possibility is $B^1 \times B^1 \times B^2$. Since $d(B^1 \times B^1 \times B^2) = 3+3+8 = 14 > 9 = n^2-7$, clearly we can rule out this possibility. Therefore, $S(\Omega, H)$ must be equivalent to the domain $D_4$. 

Suppose first that $v_1^2 > v_2^2 + v_3^2$, i.e., that $v \in \Omega_3$. Since the vector $v$ is an eigenvector of every element of $G(\Omega_3, v|w|^2)$, we see that $G(\Omega_3, v|w|^2)$ does not act transitively on $\Omega_3$. Therefore, we have $v_1 = \sqrt{v_2^2 +v_3^2} > 0$, i.e., that $v \in \partial \Omega_3 \setminus \left\{ 0 \right\}$. Since the group $G(\Omega_3)^\circ = \mathbb R_+ \times \operatorname{SO}^\circ_{1,2}$ acts transitively on $\partial \Omega_3 \setminus \left\{ 0 \right\}$, we may suppose that $v = (1,1,0.)$ 

\begin{lemma} \label{dimh} 
For the Hermitian form $\mathcal H(w, w') := (\overline{w}w', \overline{w}w', 0),$ we have
$$\dim G(\Omega_3, \mathcal H) = 3.$$ 
\end{lemma} 

\begin{proof} 
A straightforward computation of the Lie algebra of $G(\Omega_3, H)$ will prove the lemma. We momentarily denote this Lie algebra by $\mathfrak h$, and note that $\mathfrak h$ consists of all elements of $\mathfrak g(\Omega_3)$ having $(1,1,0)$ as an eigenvector. The Lie algebra $\mathfrak{g}(\Omega_3)$ is given by 
\begin{equation} \label{lieomegathree}
\mathfrak g (\Omega_3) = \mathfrak c(\mathfrak{gl}_3 (\mathbb R)) \oplus \mathfrak o_{1,2} = \left\{ \left[\begin{array}{ccc} \lambda & p & q \\ p & \lambda & r \\ q & -r & \lambda \end{array} \right]: \text{  } \lambda, p, q, r \in \mathbb R \right\}. 
\end{equation} 
Therefore, it follows that 
$$\mathfrak h = \left\{ \left[\begin{array}{ccc} \lambda & p & q \\ p & \lambda & q \\ q & -q & \lambda \end{array} \right]: \text{  } \lambda, p, q \in \mathbb R \right\}, $$
and we see that $\dim \mathfrak h = 3$ as required. 
\end{proof} 

By the above lemma we see that for $\mathfrak g = \mathfrak g(D_4)$ we have $\dim \mathfrak g_0 = 4$ (recall that $s = 1$). We also know (see \cite[Lemma 3.8 and Proposition A.3]{Isa1}) that for $\mathfrak g = \mathfrak g(D_4)$, if $v \in \partial \Omega_3 \setminus \left\{0\right\}$ we have $\mathfrak g_{1/2} = 0$ and $\dim \mathfrak g_1 = 1$. So we have 
$$d(D_4) = \dim \mathfrak g_{-1} + \dim \mathfrak g_{-1/2} + \dim \mathfrak g_0 + \dim \mathfrak g_1 = 10.$$ 
Since $d(D_4) = 10 > 9 = n^2 - 7,$ we see that $S(\Omega, H)$ is not equivalent to $D_4$, and so Case 2 contributes nothing to our classification. 
\vspace{2mm} 

\textbf{Case 3.} 
Suppose that $k = 3, n = 5$. Here, $S(\Omega, H)$ is linearly equivalent either to 
$$D_5 := \left\{ (z, w) \in \mathbb C^3 \times \mathbb C^2: \text{Im } z - \mathcal H(w,w) \in \Omega_2 \right\}, $$ 
where $\mathcal H $ is an $\Omega_2$-Hermitian form, or to 
$$D_6 := \left\{ (z, w) \in \mathbb C^3 \times \mathbb C^2: \text{Im } z - \mathcal H(w,w) \in \Omega_3 \right\}, $$ 
where $\mathcal H$ is an $\Omega_3$-Hermitian form. 

Assume $S(\Omega, H)$ is equivalent to the domain $D_5$. Then $S(\Omega, H)$ must be equivalent to either $B^1 \times B^1 \times B^3$ or $B^1 \times B^2 \times B^2$. Since $d(B^1 \times B^1 \times B^3) = 21$ and $d(B^1 \times B^2 \times B^2) = 19$, and neither of these is equal to $18 = n^2-7$, we see that consideration of the domain $D_5$ does not aid our classification.  

Suppose then that $S(\Omega, H)$ is equivalent to the domain $D_6$. With the use of Lemma \ref{dimsubspaceskewherm}, we now show that we must have either $s=1$, $s=2$ or $s=4$. Recall that $s := \dim \mathcal L$, where 
\[ \mathcal L = \left\{B \in \mathfrak{gl}_{n-k} (\mathbb C) : H(B \cdot , \cdot) + H(\cdot , B \cdot) =0 \right\}, \] 
with $H: \mathbb C^{n-k} \times \mathbb C^{n-k} \to \mathbb C^k$. Here, $\mathcal L$ consists of matrices $B \in \mathfrak{gl}_2(\mathbb C)$ such that 
\begin{align} 
H_1(B \cdot , \cdot) + H_1(\cdot , B \cdot) &=0, \nonumber \\
H_2(B \cdot , \cdot) + H_2(\cdot , B \cdot) &=0, \nonumber \\
H_3(B \cdot , \cdot) + H_3(\cdot , B \cdot) &=0. \nonumber 
\end{align} 
Writing the above relations in matrix form, and noting that since $H(w,w) \in \bar{\Omega}_3 \setminus \left\{0 \right\}$ for all non-zero $w \in \mathbb C^2$ we may assume that $H_1 = I$, we have 
\begin{align} 
B+B^*=0, \nonumber \\
H_2B+B^*H_2=0, \nonumber \\
H_3B+B^*H_3=0. \nonumber 
\end{align} 
Now consider the two vector spaces given by 
\begin{align} 
\mathcal K_2 &=\left\{B\in{\mathfrak{gl}}_2(\mathbb C): B + B^*=0,\,\,H_2B+B^*H_2=0\right\} \nonumber \\
\mathcal K_3 &=\left\{B\in{\mathfrak{gl}}_2(\mathbb C): B + B^*=0,\,\,H_3B+B^*H_3=0\right\}. \nonumber 
\end{align}
Then by Lemma \ref{dimsubspaceskewherm} we have either $\dim \mathcal K_2=2$ or $\dim \mathcal K_2=4,$ and similarly for the vector space $\mathcal K_3$. By noting that $s = \dim (\mathcal K_2 \cap \mathcal K_3)$, we have either $s=1, s=2$ or $s=4.$ Finally, the possibility of $s=0$ is excluded by observing that $i I \in \mathcal K_2 \cap K_3.$

In \cite{Isa3}, each of these scenarios was dealt with in Sections 5, 4 and 3 respectively. When $s=4$ we have $d(D_6) = 15 < 18 = n^2-7$, and when $s=2$ the action of $G(\Omega_3, H)$ on $\Omega_3$ is not transitive. When $s=1$, we see that $d(D_6) \leq 17 < 18 = n^2-7$, and so in none of these instances is any contribution made to our classification. 
\vspace{2mm} 

\textbf{Case 4.} 
Suppose that $k = 3, n = 6$. Here, $S(\Omega, H)$ is linearly equivalent either to 
$$D_7 := \left\{ (z, w) \in \mathbb C^3 \times \mathbb C^3: \text{Im } z - \mathcal H(w,w) \in \Omega_2 \right\}, $$ 
where $\mathcal H $ is an $\Omega_2$-Hermitian form, or to 
$$D_8 := \left\{ (z, w) \in \mathbb C^3 \times \mathbb C^3: \text{Im } z - \mathcal H(w,w) \in \Omega_3 \right\}, $$ 
where $\mathcal H$ is an $\Omega_3$-Hermitian form. 

Assume $S(\Omega, H)$ is equivalent to $D_7$. Then as in the previous two cases, $S(\Omega, H)$ must be biholomorphic to a product of three unit balls. The only possibilities are $B^1 \times B^1 \times B^4$, $B^1 \times B^2 \times B^3$ or $B^2 \times B^2 \times B^2$, none of which have automorphism group of dimension $n^2 - 7 = 29$. 

So $S(\Omega, H)$ must be equivalent to $D_8$. By (\ref{estim2}) we have $s + \dim \mathfrak{g}(\Omega) \geq 11$. Since $\dim \mathfrak{g}(\Omega_3) = 4$, we see that $s \ge 7$. On the other hand, by (\ref{ests}) we have $s \leq 9$. We show now that $s$ cannot equal $7$ or $8$, and so we must have $s=9$. By a similar argument to that in the previous case, the two vector spaces $\mathcal K_2$ and $\mathcal K_3$ are in this instance given by 
\begin{align} 
\mathcal K_2 &=\left\{B\in{\mathfrak{gl}}_3(\mathbb C): B + B^*=0,\,\,H_2B+B^*H_2=0\right\} \nonumber \\
\mathcal K_3 &=\left\{B\in{\mathfrak{gl}}_3(\mathbb C): B + B^*=0,\,\,H_3B+B^*H_3=0\right\}. \nonumber 
\end{align}
By Lemma \ref{dimsubspaceskewherm} we have that $\dim \mathcal K_2 = 3, 5$ or $9$ (recall that pairs of distinct eigenvalues are counted with multiplicity). Similarly, $\dim \mathcal K_3 = 3, 5$ or $9$. Noting that $s= \dim (\mathcal K_2 \cap \mathcal K_3)$, clearly $s = 1, 2, 3, 5$ or $9$. 

Let $\mathcal H = (\mathcal H_1, \mathcal H_2, \mathcal H_3)$ and {\bf H} be a positive-definite linear combination of $\mathcal H_1, \mathcal H_2, \mathcal H_3$. After a linear change of $w$-variables, we can diagonalise {\bf H} as ${\bf H} (w,w) = ||w||^2$. Since $s=9$, by Lemma \ref{dimsubspaceskewherm} each of the $\mathbb C$-valued Hermitian forms $\mathcal H_1, \mathcal H_2, \mathcal H_3$ is proportional to ${\bf H}$. Thus we have $\mathcal H (w,w) = v||w||^2$, where $v = (v_1, v_2, v_3)$ is a vector in $\mathbb R^3$ satisfying $v_1^2 \geq v_2^2 + v_3^2$, $v_1 > 0$. 

Suppose first that $v_1^2 > v_2^2 + v_3^2$, i.e., that $v \in \Omega_3$. Since the vector $v$ is an eigenvector of every element of $G(\Omega_3, v|w|^2)$, we see that $G(\Omega_3, v|w|^2)$ does not act transitively on $\Omega_3$. Therefore, we have $v_1 = \sqrt{v_2^2 +v_3^2} > 0$, i.e., that $v \in \partial \Omega_3 \setminus \left\{ 0 \right\}$. As the group $G(\Omega_3)^\circ = \mathbb R_+ \times \text{SO}^{\circ}_{1,2}$ acts transitively on $\partial \Omega_3 \setminus \left\{ 0 \right\}$, we can suppose that $v=(1,1,0)$, so $\mathcal H (w,w) = (||w||^2, ||w||^2, 0)$. Here, the domain $D_8$ coincides with the domain $\widetilde{D}_6$ with $N = 3$ (see \cite[Lemma 3.4]{Isa3}). Therefore, we see that for $\mathfrak{g} = \mathfrak{g}(D_8)$ we have $\mathfrak{g}_{1/2} = 0$, and by \cite[Lemma 3.5]{Isa3} we see that for $\mathfrak{g} = \mathfrak{g}(D_8)$ we have $\dim \mathfrak{g}_{1} = 1.$ Furthermore, for $w \in \mathbb C^3$, the proof of Lemma \ref{dimh} gives us 
$$\dim G(\Omega_3, (||w||^2, ||w||^2, 0)) = 3,$$
and we see that $\dim \mathfrak g_0 = 12$ (since $s=9$). Therefore, we have 
$$d(D_8) = \dim \mathfrak{g}_{-1} + \dim \mathfrak{g}_{-1/2} + \dim \mathfrak{g}_{0} + \dim \mathfrak{g}_{1/2} + \dim \mathfrak{g}_{1} = 22 < 29 = n^2 - 7.$$ 
This shows that $S(\Omega, H)$ cannot be equivalent to $D_8$, so Case 4 contributes nothing to our classification. 
\vspace{2mm} 

\textbf{Case 5.} 
Suppose that $k=3, n=7$. Here, $S(\Omega, H)$ is linearly equivalent either to 
$$D_9 := \left\{ (z, w) \in \mathbb C^3 \times \mathbb C^4: \text{Im } z - \mathcal H(w,w) \in \Omega_2 \right\}, $$ 
where $\mathcal H $ is an $\Omega_2$-Hermitian form, or to 
$$D_{10} := \left\{ (z, w) \in \mathbb C^3 \times \mathbb C^4: \text{Im } z - \mathcal H(w,w) \in \Omega_3 \right\}, $$ 
where $\mathcal H$ is an $\Omega_3$-Hermitian form. 

By (\ref{estim2}) we have $s+ \dim \mathfrak g(\Omega) \ge 20$. On the other hand, $s \le 16$ by (\ref{ests}). Since $\dim \mathfrak g (\Omega_2)=3$ and $\dim \mathfrak g (\Omega_3)=4$, it follows that $\Omega$ is linearly equivalent to $\Omega_3$ and $s=16.$ In particular, $S(\Omega, H)$ can only be linearly equivalent to the domain $D_{10}$. 

We proceed in the same manner as the previous case. Let $\mathcal H = (\mathcal H_1, \mathcal H_2, \mathcal H_3)$ and {\bf H} be a positive-definite linear combination of $\mathcal H_1, \mathcal H_2, \mathcal H_3$. After a linear change of $w$-variables, we can diagonalise {\bf H} as ${\bf H} (w,w) = ||w||^2$. Since $s=16$, by Lemma \ref{dimsubspaceskewherm} each of the $\mathbb C$-valued Hermitian forms $\mathcal H_1, \mathcal H_2, \mathcal H_3$ is proportional to ${\bf H}$. Thus we have $\mathcal H (w,w) = v||w||^2$, where $v = (v_1, v_2, v_3)$ is a vector in $\mathbb R^3$ satisfying $v_1^2 \geq v_2^2 + v_3^2$, $v_1 > 0$. 

We will suppose first that $v_1^2 > v_2^2 + v_3^2$, i.e., that $v \in \Omega_3$. Since the vector $v$ is an eigenvector of every element of $G(\Omega_3, v|w|^2)$, we see that $G(\Omega_3, v|w|^2)$ does not act transitively on $\Omega_3$. Therefore, we have $v_1 = \sqrt{v_2^2 +v_3^2} > 0$, i.e., that $v \in \partial \Omega_3 \setminus \left\{ 0 \right\}$. As the group $G(\Omega_3)^\circ = \mathbb R_+ \times \text{SO}^{\circ}_{1,2}$ acts transitively on $\partial \Omega_3 \setminus \left\{ 0 \right\}$, we can suppose that $v=(1,1,0)$, so $\mathcal H (w,w) = (||w||^2, ||w||^2, 0)$. In this case, the domain $D_{10}$ coincides with the domain $\widetilde{D}_6$ with $N = 4$ (see \cite[Lemma 3.4]{Isa3}). As in the previous case, we see that for $\mathfrak{g} = \mathfrak{g}(D_{10})$ we have $\mathfrak{g}_{1/2} = 0$ and $\dim \mathfrak{g}_1 = 1.$ Furthermore, for $w \in \mathbb C^4$, the proof of Lemma \ref{dimh} gives us 
$$\dim G(\Omega_3, (||w||^2, ||w||^2, 0)) = 3,$$
and we see that $\dim \mathfrak g_0 = 19$ (since $s=16$). Therefore, we have
$$d(D_8) = \dim \mathfrak{g}_{-1} + \dim \mathfrak{g}_{-1/2} + \dim \mathfrak{g}_{0} + \dim \mathfrak{g}_{1/2} + \dim \mathfrak{g}_{1} = 31 < 42 = n^2 - 7.$$ 
This shows that $S(\Omega, H)$ cannot be equivalent to $D_{10}$, so Case 5 makes no contributions to the classification.
\vspace{2mm} 

\textbf{Case 6.} 
Suppose that $k=4, n=4$. In this case, after a linear change of variables, $S(\Omega, H)$ is one of the domains 
$$\left\{ z \in \mathbb C^4 : \text{Im } z \in \Omega_4 \right\},$$ 
$$\left\{ z \in \mathbb C^4 : \text{Im } z \in \Omega_5 \right\},$$ 
$$\left\{ z \in \mathbb C^4 : \text{Im } z \in \Omega_6 \right\},$$ 
and therefore is biholomorphic either to $B^1 \times B^1 \times B^1 \times B^1$, or to $B^1 \times T_3$, where $T_3$ is the domain 
$$T_3 = \left\{(z_1, z_2, z_3) \in \mathbb C^3 : (\operatorname{Im }z_1)^2 -(\operatorname{Im }z_2)^2 - (\operatorname{Im }z_3)^2 >0, \hspace{2mm} \operatorname{Im }z_1 >0 \right\},$$ 
or to $T_4$, where $T_4$ is the domain 
\begin{align}
T_4 = \big\{(z_1, z_2, z_3, z_4) \in \mathbb C^4 : (\operatorname{Im }z_1)^2 -(\operatorname{Im }&z_2)^2 - (\operatorname{Im }z_3)^2 - (\operatorname{Im }z_4)^2>0, \nonumber \\ 
&\operatorname{Im }z_1 >0 \big\}. \nonumber 
\end{align} 
The dimensions of the respective automorphism groups of these domains are 12, 13 and 15. Each of these numbers is greater than $9 = n^2 - 7$, and so we see that Case 6 contributes nothing to our classification. 
\vspace{2mm} 

\textbf{Case 7.} 
Suppose that $k=4, n=5$. Then $S(\Omega, H)$ is linearly equivalent to either 
$$D_{11} := \left\{ (z, w) \in \mathbb C^4 \times \mathbb C: \text{Im } z - v|w|^2 \in \Omega_4 \right\}, $$ 
where $v = (v_1, v_2, v_3, v_4)$ is a vector in $\mathbb R^4$ with non-negative entries, or 
$$D_{12} := \left\{ (z, w) \in \mathbb C^4 \times \mathbb C: \text{Im } z - v|w|^2 \in \Omega_5 \right\}, $$ 
where $v = (v_1, v_2, v_3, v_4)$ is a vector in $\mathbb R^4$ satisfying $v \in \bar{\Omega}_5 \setminus \left\{ 0 \right\}$, or 
$$D_{13} := \left\{ (z, w) \in \mathbb C^4 \times \mathbb C: \text{Im } z - v|w|^2 \in \Omega_6 \right\}, $$ 
where $v = (v_1, v_2, v_3, v_4)$ is a vector in $\mathbb R^4$ satisfying $v \in \bar{\Omega}_6 \setminus \left\{ 0 \right\}$, i.e., $v_1^2 \geq v_2^2 + v_3^2 + v_4^2, v_1 > 0.$ 

Since $s=1$, by inequality (\ref{estim2}) we see that $\dim \mathfrak{g}(\Omega) \geq 5$. Therefore $S(\Omega, H)$ can only be linearly equivalent to either $D_{12}$ or $D_{13}$. Let us begin with the second possibility. If $S(\Omega, H)$ is equivalent to $D_{13}$, then assume firstly that $v_1^2 > v_2^2 + v_3^2 + v_4^2$, i.e., that $v \in \Omega_6$. Since the vector $v$ is an eigenvector of every element of $G(\Omega_6, v|w|^2)$, we see that $G(\Omega_6, v|w|^2)$ does not act transitively on $\Omega_6$. Therefore, we have $v_1 = \sqrt{v_2^2 +v_3^2 +v_4^2} > 0$, i.e., that $v \in \partial \Omega_6 \setminus \left\{ 0 \right\}$. As the group $G(\Omega_6)^\circ = \mathbb R_+ \times \text{SO}^{\circ}_{1,3}$ acts transitively on $\partial \Omega_6 \setminus \left\{ 0 \right\}$, we can suppose that $v=(1,1,0,0)$, so $v|w|^2 = (|w|^2, |w|^2, 0, 0)$. 

\begin{lemma} 
For the Hermitian form $\mathcal H(w, w') := (\overline{w}w', \overline{w}w', 0, 0),$ we have
$$\dim G(\Omega_6, \mathcal H) = 5.$$ 
\end{lemma} 
\begin{proof} 
A straightforward computation of the Lie algebra of $G(\Omega_6, H)$ will prove the lemma. We momentarily denote this Lie algebra by $\mathfrak h$, and note that $\mathfrak h$ consists of all elements of $\mathfrak g(\Omega_6)$ having $(1,1,0,0)$ as an eigenvector. The Lie algebra $\mathfrak{g}(\Omega_6)$ is given by 
$$\mathfrak g (\Omega_6) = \mathfrak c(\mathfrak{gl}_4 (\mathbb R)) \oplus \mathfrak o_{1,3} = \left\{ \left[\begin{array}{cccc} \lambda & p & q & r \\ p & \lambda & s & t \\ q & -s & \lambda & y \\ r & -t & -y & \lambda \end{array} \right]: \text{  } \lambda, p, q, r, s, t, y \in \mathbb R \right\}.$$
Therefore, it follows that 
$$\mathfrak h = \left\{ \left[\begin{array}{cccc} \lambda & p & q & r \\ p & \lambda & q & r \\ q & -q & \lambda & y \\ r & -r & -y & \lambda \end{array} \right]: \text{  } \lambda, p, q, r, y \in \mathbb R \right\}, $$
and we see that $\dim \mathfrak h = 5$ as required. 
\end{proof} 
 By the above lemma and the equality in (\ref{estim1}) we see that for $\mathfrak g = \mathfrak g(D_{13})$ we have $\dim \mathfrak g_0 = 6$ (recall that $s = 1$). Further, the domain $D_{13}$ coincides with the domain $\widetilde{D}_{13}$ with $N=1$ (see \cite[Lemma 4.3]{Isa3}), so we see that for $\mathfrak{g} = \mathfrak{g}(D_{13})$ we have $\mathfrak{g}_{1/2} = 0$. Now using these values for $\dim \mathfrak g_0$ and $\dim \mathfrak g_{1/2}$ along with the second inequality in (\ref{estimm}), we see that 
$$d(D_{13}) = \dim \mathfrak{g}_{-1} + \dim \mathfrak{g}_{-1/2} + \dim \mathfrak{g}_{0} + \dim \mathfrak{g}_{1/2} + \dim \mathfrak{g}_{1} \leq 16 < 18 = n^2 - 7,$$ 
showing no contribution to the classification. 

Now assume that $S(\Omega, H)$ is equivalent to $D_{12}$. To begin with, consider the boundary set $\partial \Omega_5 \setminus \left\{ 0 \right\}$, which can be described 
\begin{align} 
\partial \Omega_5 \setminus \left\{0 \right\} &= \left\{ (v_1,v_2,v_3,v_4) \in \mathbb R^4: v_1^2 \geq v_2^2 + v_3^2,\,\, v_1 > 0,\,\, v_4=0 \right\} \nonumber \\ 
& \hspace{0.5cm} \cup \left\{ (v_1,v_2,v_3,v_4) \in \mathbb R^4: v_1^2 = v_2^2 + v_3^2,\,\, v_1 \geq 0,\,\, v_4>0 \right\}. \nonumber 
\end{align} 
We can further break up this boundary set into four components, which are invariant under the action of $G(\Omega_5)^\circ$, and which we denote by $C_1, C_2, C_3$ and $C_4$. Describing each of these components, we have 
\begin{align}
C_1&:= \left\{ (v_1,v_2,v_3,v_4) \in \mathbb R^4: v_1^2 > v_2^2 + v_3^2,\,\, v_1 > 0,\,\, v_4=0 \right\}, \nonumber \\
C_2&:= \left\{ (v_1,v_2,v_3,v_4) \in \mathbb R^4: v_1^2 = v_2^2 + v_3^2,\,\, v_1 > 0,\,\, v_4=0 \right\}, \nonumber \\
C_3&:= \left\{ (v_1,v_2,v_3,v_4) \in \mathbb R^4: v_1^2 = v_2^2 + v_3^2,\,\, v_1 > 0,\,\, v_4>0 \right\}, \text{ and } \nonumber \\
C_4&:= \left\{ (v_1,v_2,v_3,v_4) \in \mathbb R^4: v_1^2 = v_2^2 + v_3^2,\,\, v_1 = 0,\,\, v_4>0 \right\}. \nonumber 
\end{align} 
Assume first that $v \in C_1$, i.e., that $v \in \Omega_3 \times \left\{ 0 \right\}$ (recall that $\Omega_3 := \Lambda_3$, the Lorentz cone in $\mathbb R^3$). In this situation, we have the following lemma. 
\begin{lemma} \label{componentone} 
If $v \in \Omega_3 \times \left\{0 \right\}$, for $\mathfrak g = \mathfrak g(D_{12})$ we have $\mathfrak g_{1/2}=0$. 
\end{lemma} 
\begin{proof} 
Since the group $G(\Omega_3)^\circ = \mathbb R_+ \times \operatorname{SO}^{\circ}_{1,2}$ acts transitively on $\Omega_3$, we may suppose that $v=(1,0,0,0)$. We will apply Theorem \ref{descrg1/2} to the cone $\Omega_5$ and the $\Omega_5$-Hermitian form 
\[\mathcal H (w, w')= (\bar ww', 0,0,0). \] 
Let $\Phi: \mathbb C^4 \to \mathbb C$ be a $\mathbb C$-linear map given by 
\[\Phi(z_1,z_2,z_3,z_4) = \varphi_1 z_1 + \varphi_2 z_2 + \varphi_3 z_3 + \varphi_4 z_4 \]
where $\varphi_j \in \mathbb C$. Fixing $\mathbf w \in \mathbb C$, for $x \in \mathbb R^4$ we compute 
\[\mathcal H (\mathbf w, \Phi(x)) = (\bar{\mathbf w}(\varphi_1 x_1 + \varphi_2 x_2 + \varphi_3 x_3 + \varphi_4 x_4),0,0,0). \] 
Then from formula (\ref{Phiw0}) we see
\[\Phi_{\mathbf w}(x)=(\operatorname{Im}(\bar{\mathbf w} \varphi_1)x_1 + \operatorname{Im}(\bar{\mathbf w} \varphi_2)x_2 + \operatorname{Im}(\bar{\mathbf w} \varphi_3)x_3 + \operatorname{Im}(\bar{\mathbf w} \varphi_4)x_4, 0,0,0). \] 
Now, since $\mathfrak g(\Omega_5) = (\mathfrak c(\mathfrak{gl}_3(\mathbb R)) \oplus \mathfrak o_{1,2}) \oplus \mathbb R$ consists of all matrices of the form 
\[ \left[\begin{array}{cccc} \lambda & p & q & 0 \\ p & \lambda & r & 0 \\ q & -r & \lambda & 0 \\ 0 & 0 & 0 & \mu \end{array} \right], \text{  } \lambda, \mu, p, q, r \in \mathbb R. \] 
Therefore, the condition that the map $\Phi_{\mathbf w}$ lies in $\mathfrak g(\Omega_5)$ for every $\mathbf w \in \mathbb C$ immediately yields 
\[ \operatorname{Im}(\bar{\mathbf w} \varphi_1) \equiv 0,\,\, \operatorname{Im}(\bar{\mathbf w} \varphi_2) \equiv 0, \,\,\operatorname{Im}(\bar{\mathbf w} \varphi_3) \equiv 0 \,\, \text{ and } \, \operatorname{Im}(\bar{\mathbf w} \varphi_4) \equiv 0, \] 
which implies $\Phi = 0$. By formula (\ref{cond1}) we then see that $\mathfrak g_{1/2} = 0$ as required. 
\end{proof} 

Now assume that $v \in C_2$, i.e., that $v \in \partial \Omega_3 \setminus \left\{0 \right\} \times \left\{0 \right\}$. In this situation, we have the following lemma, which is analogous to the previous one. 
\begin{lemma} \label{componenttwo} 
If $v \in \partial \Omega_3 \setminus \left\{0 \right\} \times \left\{0 \right\}$, for $\mathfrak g = \mathfrak g(D_{12})$ we have $\mathfrak g_{1/2}=0$. 
\end{lemma} 
\begin{proof} 
Since the group $G(\Omega_3)^\circ = \mathbb R_+ \times \operatorname{SO}^{\circ}_{1,2}$ acts transitively on $\partial \Omega_3 \setminus \left\{0 \right\}$, we may suppose that $v=(1,1,0,0)$. We apply Theorem \ref{descrg1/2} to the cone $\Omega_5$ and the $\Omega_5$-Hermitian form 
\[\mathcal H (w, w')= (\bar ww', \bar ww',0,0). \] 
Let $\Phi: \mathbb C^4 \to \mathbb C$ be a $\mathbb C$-linear map given by 
\[\Phi(z_1,z_2,z_3,z_4) = \varphi_1 z_1 + \varphi_2 z_2 + \varphi_3 z_3 + \varphi_4 z_4 \]
where $\varphi_j \in \mathbb C$. Fixing $\mathbf w \in \mathbb C$, for $x \in \mathbb R^4$ we compute 
\begin{align} 
\mathcal H (\mathbf w, \Phi(x)) = (\bar{\mathbf w} & (\varphi_1 x_1 + \varphi_2 x_2 + \varphi_3 x_3 + \varphi_4 x_4), \nonumber  \\ 
& \bar{\mathbf w}(\varphi_1 x_1 + \varphi_2 x_2 + \varphi_3 x_3 + \varphi_4 x_4),0,0). \nonumber 
\end{align} 
Then from formula (\ref{Phiw0}) we see
\begin{align}
\Phi_{\mathbf w}(x)=(\operatorname{Im} & (\bar{\mathbf w} \varphi_1)x_1 + \operatorname{Im}(\bar{\mathbf w} \varphi_2)x_2 + \operatorname{Im}(\bar{\mathbf w} \varphi_3)x_3 + \operatorname{Im}(\bar{\mathbf w} \varphi_4)x_4, \nonumber \\ 
& \operatorname{Im}(\bar{\mathbf w} \varphi_1)x_1 + \operatorname{Im}(\bar{\mathbf w} \varphi_2)x_2 + \operatorname{Im}(\bar{\mathbf w} \varphi_3)x_3 + \operatorname{Im}(\bar{\mathbf w} \varphi_4)x_4,0,0). \nonumber 
\end{align} 
Now, since $\mathfrak g(\Omega_5) = (\mathfrak c(\mathfrak{gl}_3(\mathbb R)) \oplus \mathfrak o_{1,2}) \oplus \mathbb R$ consists of all matrices of the form 
\[ \left[\begin{array}{cccc} \lambda & p & q & 0 \\ p & \lambda & r & 0 \\ q & -r & \lambda & 0 \\ 0 & 0 & 0 & \mu \end{array} \right], \text{  } \lambda, \mu, p, q, r \in \mathbb R. \] 
Therefore, the condition that the map $\Phi_{\mathbf w}$ lies in $\mathfrak g(\Omega_5)$ for every $\mathbf w \in \mathbb C$ immediately yields 
\[ \operatorname{Im}(\bar{\mathbf w} \varphi_1) \equiv 0,\,\, \operatorname{Im}(\bar{\mathbf w} \varphi_2) \equiv 0, \,\,\operatorname{Im}(\bar{\mathbf w} \varphi_3) \equiv 0 \,\, \text{ and } \, \operatorname{Im}(\bar{\mathbf w} \varphi_4) \equiv 0, \] 
which implies $\Phi = 0$. By formula (\ref{cond1}) we then see that $\mathfrak g_{1/2} = 0$ as required. 
\end{proof} 

Now assume that $v \in C_3$, i.e., that $v \in \partial \Omega_3 \setminus \left\{0 \right\} \times \mathbb R_+$. In this situation, we have the following lemma, which is analogous to the previous two. 
\begin{lemma} \label{componentthree} 
If $v \in \partial \Omega_3 \setminus \left\{0 \right\} \times \mathbb R_+$, for $\mathfrak g = \mathfrak g(D_{12})$ we have $\mathfrak g_{1/2}=0$. 
\end{lemma} 
\begin{proof} 
Since the group $G(\Omega_3)^\circ = \mathbb R_+ \times \operatorname{SO}^{\circ}_{1,2}$ acts transitively on $\partial \Omega_3 \setminus \left\{0 \right\}$, we may suppose that $v=(1,1,0,1)$. We again apply Theorem \ref{descrg1/2} to the cone $\Omega_5$ and the $\Omega_5$-Hermitian form 
\[\mathcal H (w, w')= (\bar ww', \bar ww',0, \bar ww'). \] 
Let $\Phi: \mathbb C^4 \to \mathbb C$ be a $\mathbb C$-linear map given by 
\[\Phi(z_1,z_2,z_3,z_4) = \varphi_1 z_1 + \varphi_2 z_2 + \varphi_3 z_3 + \varphi_4 z_4 \]
where $\varphi_j \in \mathbb C$. Fixing $\mathbf w \in \mathbb C$, for $x \in \mathbb R^4$ we compute 
\begin{align} 
\mathcal H (\mathbf w, \Phi(x)) = (\bar{\mathbf w} & (\varphi_1 x_1 + \varphi_2 x_2 + \varphi_3 x_3 + \varphi_4 x_4), \nonumber  \\ 
& \bar{\mathbf w}(\varphi_1 x_1 + \varphi_2 x_2 + \varphi_3 x_3 + \varphi_4 x_4), 0, \nonumber \\ 
& \hspace{4mm} \bar{\mathbf w}(\varphi_1 x_1 + \varphi_2 x_2 + \varphi_3 x_3 + \varphi_4 x_4)). \nonumber 
\end{align} 
Then from formula (\ref{Phiw0}) we see
\begin{align}
\Phi_{\mathbf w}(x)=(\operatorname{Im} & (\bar{\mathbf w} \varphi_1)x_1 + \operatorname{Im}(\bar{\mathbf w} \varphi_2)x_2 + \operatorname{Im}(\bar{\mathbf w} \varphi_3)x_3 + \operatorname{Im}(\bar{\mathbf w} \varphi_4)x_4, \nonumber \\ 
& \operatorname{Im}(\bar{\mathbf w} \varphi_1)x_1 + \operatorname{Im}(\bar{\mathbf w} \varphi_2)x_2 + \operatorname{Im}(\bar{\mathbf w} \varphi_3)x_3 + \operatorname{Im}(\bar{\mathbf w} \varphi_4)x_4,0, \nonumber \\
& \hspace{6mm} \operatorname{Im}(\bar{\mathbf w} \varphi_1)x_1 + \operatorname{Im}(\bar{\mathbf w} \varphi_2)x_2 + \operatorname{Im}(\bar{\mathbf w} \varphi_3)x_3 + \operatorname{Im}(\bar{\mathbf w} \varphi_4)x_4). \nonumber 
\end{align} 
Now, since $\mathfrak g(\Omega_5) = (\mathfrak c(\mathfrak{gl}_3(\mathbb R)) \oplus \mathfrak o_{1,2}) \oplus \mathbb R$ consists of all matrices of the form 
\[ \left[\begin{array}{cccc} \lambda & p & q & 0 \\ p & \lambda & r & 0 \\ q & -r & \lambda & 0 \\ 0 & 0 & 0 & \mu \end{array} \right], \text{  } \lambda, \mu, p, q, r \in \mathbb R. \] 
Therefore, the condition that the map $\Phi_{\mathbf w}$ lies in $\mathfrak g(\Omega_5)$ for every $\mathbf w \in \mathbb C$ immediately yields 
\[ \operatorname{Im}(\bar{\mathbf w} \varphi_1) \equiv 0,\,\, \operatorname{Im}(\bar{\mathbf w} \varphi_2) \equiv 0, \,\,\operatorname{Im}(\bar{\mathbf w} \varphi_3) \equiv 0 \,\, \text{ and } \, \operatorname{Im}(\bar{\mathbf w} \varphi_4) \equiv 0, \] 
which implies $\Phi = 0$. By formula (\ref{cond1}) we then see that $\mathfrak g_{1/2} = 0$ as required. 
\end{proof} 
We see from the above three lemmas that for the components $C_1, C_2$ and $C_3$ we have $\mathfrak g_{1/2}=0$. Then by estimate (\ref{estim 8}), the second inequality in (\ref{estimm}) and the above three lemmas, we see that in each of these cases 
\[ d(D_{12}) \le 16 < 18 = n^2-7 \] 
(recall that $s=1$). This shows that in the cases of these components, $S(\Omega, H)$ cannot be equivalent to $D_{12}$, so no new contributions are made to our classification. 

Lastly, let $v \in C_4$, i.e., that $v \in \left\{(0,0,0) \right\} \times \mathbb R_+$. Since $G(\Omega_5)^{\circ}$ clearly acts transitively on this set, we may assume that $v=(0,0,0,1)$. Then $S(\Omega, H)$ is equivalent to the domain 
\begin{align}
B^2 \times T_3 = \big\{(z_1, z_2, z_3, z_4) \in \mathbb C^4 : (\operatorname{Im }z_1)^2 -(\operatorname{Im }&z_2)^2 - (\operatorname{Im }z_3)^2 >0, \,\, \operatorname{Im} z_1 >0, \nonumber \\ 
&\operatorname{Im }z_4 - |w|^2 >0 \big\}. \nonumber 
\end{align} 
Since $d(B^2 \times T_3) = 10 + 8 = 18 = n^2-7$, we see that Case $7$ contributes the product $B^2 \times T_3$ to the classification of homogeneous Kobayashi-hyperbolic manifolds with automorphism group dimension $n^2-7$. 
\vspace{2mm} 

\textbf{Case 8.} 
Suppose that $k=5$ and $n=5$. Then by inequality (\ref{estim2}) we see that in this situation we have $\dim \mathfrak g(\Omega) \geq 8$. Therefore, $S(\Omega, H)$ is equivalent to one of the domains 
$$\left\{z \in \mathbb C^5: \operatorname{Im }z \in \Omega_9 \right\},$$ 
$$\left\{z \in \mathbb C^5: \operatorname{Im }z \in \Omega_{10} \right\},$$ 
and therefore is biholomorphic either to $B^1 \times T_4$, where 
\begin{align}
T_4 = \big\{(z_1, z_2, z_3, z_4) \in \mathbb C^4 : (\operatorname{Im }z_1)^2 -(\operatorname{Im }&z_2)^2 - (\operatorname{Im }z_3)^2 - (\operatorname{Im }z_4)^2>0, \nonumber \\ 
&\operatorname{Im }z_1 >0 \big\}, \nonumber 
\end{align} 
or to $T_5$, where 
\begin{align}
T_5 = \big\{(z_1, z_2, z_3, z_4, z_5) \in \mathbb C^5 : (\operatorname{Im }&z_1)^2 -(\operatorname{Im }z_2)^2 - (\operatorname{Im }z_3)^2 \nonumber \\
&- (\operatorname{Im }z_4)^2 - (\operatorname{Im }z_5)^2 >0, \operatorname{Im }z_1 >0 \big\}. \nonumber 
\end{align} 
In the latter case, the dimension of the automorphism group of this domain is $d(T_5) = 21 > 18 = n^2-7$, so no contribution is made. However, we see that $d(B^1 \times T_4) = 3 + 15 = 18 = n^2-7$, and so Case 4 contributes $B^1 \times T_4$ to the classification of homogeneous hyperbolic manifolds with automorphism group dimension $n^2-7.$ This completes the proof. 

\section{Proof of Theorem 2} 

Let $M$ be a homogeneous Kobayashi-hyperbolic manifold of dimension $n$. By Theorem \cite{N}, the manifold $M$ is biholomorphic to an affinely homogeneous Siegel domain of the second kind $S(\Omega, H)$. As in the previous section, we take $n \ge 4$ and $k \ge 2.$ We can use the following lemma to rule out a large number of remaining possibilities. 

\begin{lemma}
    For the following values of $n$ and $k$, we cannot have $d(S(\Omega, H))= n^2-8$: 
        \begin{enumerate}
            \item $n \ge 7, k \ge 4$, 
            \item $n \ge 8, k=3$, 
            \item $n=6, k=4$, 
            \item $n=6, k=5$. 
        \end{enumerate}
\end{lemma}

\begin{proof} 
To prove the lemma, we will show that for $n \geq 7, k \geq 4$, as well as for $n \geq 8, k = 3$ and the two cases $n=6, k=4$ and $n=6, k=5$, the right-hand side of the inequality 
$$ d(S(\Omega, H)) \leq \frac{3k^2}{2} - \left(2n + \frac{5}{2} \right) k + n^2 +4n + 1 $$
is strictly less that $n^2 - 8$. That is, for these $k, n$ the following holds: 
\begin{equation} \label{eight} 
\frac{3k^2}{2} - \left(2n + \frac{5}{2} \right) k  + 4n + 9 < 0. 
\end{equation} 
To see this, consider the quadratic function 
$$ \varphi(t) := \frac{3t^2}{2} - \left(2n + \frac{5}{2} \right) t + 4n + 9. $$ 
The discriminant of $\varphi$ is given by 
$$ \mathcal{D} := 4n^2 - 14n - \frac{191}{4}, $$ 
which is positive for $n \geq 6$. The zeroes of $\varphi$ are given by 
$$t_1 := \frac{2n + \frac{5}{2} - \sqrt{\mathcal D}}{3},$$
$$t_2 := \frac{2n + \frac{5}{2} + \sqrt{\mathcal D}}{3}.$$ 

To prove the lemma, it suffices to show that: (i) $t_2 > n$ for $n \geq 7$, (ii) $t_1 < 4$ for $n \geq 7$, (iii) $t_1 < 3$ for $n \geq 8$, and lastly (iv) each pair $n=6, k=4$ and $n=6, k=5$ satisfies (\ref{eight}).  Beginning with the inequality $t_2 > n$, we have 
$$n - \frac{5}{2} < \sqrt{\mathcal D},$$
or, equivalently, that 
$$n^2 - 3n - 18 > 0,$$ 
which clearly holds for $n \geq 7$. Now considering $t_1 < 4$, we see that 
$$2n - \frac{19}{2} < \sqrt{\mathcal D}, $$ 
or, equivalently, that 
$$n > \frac{23}{4}, $$ 
which holds for $n \geq 7.$ Lastly, the inequality $t_1 < 3$ implies that 
$$2n - \frac{13}{2} < \sqrt{\mathcal D}, $$
or, equivalently, that 
$$n > \frac{15}{2}, $$ 
which holds for $n \geq 8$. 

Finally, the pairs $n=6, k=4$ and $n=6, k=5$ clearly satisfy (\ref{eight}). 
\end{proof} 

By the above lemma, we can prove the theorem by considering the following nine cases: 
\begin{enumerate}
\item $k = 2, n \geq 4.$ 
\item $k = 3, n = 4.$
\item $k = 3, n = 5.$ 
\item $k = 3, n = 6.$ 
\item $k = 3, n = 7.$
\item $k = 4, n = 4.$ 
\item $k = 4, n = 5.$ 
\item $k = 5, n = 5.$
\item $k = 6, n = 6.$
\end{enumerate}
We now begin by considering each case.

\textbf{Case 1.} 
Suppose that $k = 2, n \geq 4.$ Recall from the previous section that $H= (H_1, H_2)$ is a pair of Hermitian forms on $\mathbb C^{n-2}$, where we may take $H_1$ to be positive definite. They are simultaneously diagonalised as 
\[H_1(w,w) = ||w||^2, \hspace{5mm} H_2(w,w) = \sum_{j=1}^{n-2} \lambda_j |w_j|^2. \] 
Recall further that $H_2$ has at least one pair of distinct eigenvalues, and that $m \geq 1$ denotes the number of pairs of these eigenvalues. 

As $\dim \mathfrak{g}(\Omega) = 2$, inequality (\ref{estim2}) yields 
\begin{equation} 
s \geq n^2 - 4n - 6. \label{equat1} 
\end{equation} 
On the other hand, by inequality (\ref{ests}),
\begin{equation} 
s \leq n^2 - 4n + 4.  \label{equat2} 
\end{equation} 
By Lemma \ref{dimsubspaceskewherm} the exact value of $s$ is given by 
$$s = n^2 - 4n + 4 - 2m, $$
which implies $m = 1, 2, 3, 4$ or $5$. The values $m = 1, 2, 3, 4$ are treated as in the previous section, and contribute no additional domains. However, the possibility of $m = 5$ leads to two additional subcases: (g) where $n=6$ with $\lambda_1 \neq \lambda_2 \neq \lambda_3 = \lambda_4$ where $\lambda_1 \ne \lambda_3$, and (h) where $n = 8$ with $\lambda_1 \neq \lambda_2 = \lambda_3 = \lambda_4 = \lambda_5 = \lambda_6.$ When $k = 2$ and $n \geq 4$, $S(\Omega, H)$ is biholomorphic to a product of two unit balls $B^l \times B^{n-l}$ for $1 \leq l \leq n-1$, and the dimension of its automorphism group is given by 
$$d(B^l \times B^{n-l}) = 2l^2 - 2nl + n^2 + 2n.$$ 
Setting the right-hand side equal to $n^2-8$, we see that $l$ is integer-valued only in the case of $n=8.$ In this case, $l = 2,$ and so Case 1 contributes the product $B^2 \times B^6$ to the classification, with $d(B^2 \times B^6) = 8 + 48 = 56 = n^2 - 8.$ 
\vspace{2mm} 

\textbf{Case 2.} 
Suppose that $k = 3, n = 4$. Then $S(\Omega, H)$ is equivalent to either
$$D_3 := \left\{ (z,w) \in \mathbb C^3 \times \mathbb C : \text{Im } z - v|w|^2 \in \Omega_2 \right\},$$ 
where $v = (v_1, v_2, v_3)$ is a vector in $\mathbb R^3$ with non-negative entries, or 
$$D_4 := \left\{ (z,w) \in \mathbb C^3 \times \mathbb C : \text{Im } z - v|w|^2 \in \Omega_3 \right\},$$ 
where $v = (v_1, v_2, v_3)$ is a vector in $\mathbb R^3$ satisfying $v_1^2 \geq v_2^2 + v_3^2,$ $v_1 > 0.$ 
As in the previous section, we begin by assuming that $S(\Omega, H)$ is equivalent to the domain $D_3$. Then $S(\Omega, H)$ can only be biholomorphic to the product $B^1 \times B^1 \times B^2$. This cannot occur, since $d(B^1 \times B^1 \times B^2) = 14 > 8 = n^2-8$.

Therefore, assume $S(\Omega, H)$ is equivalent to $D_4.$ Recall from the previous section that if $v \in \Omega_3$, then the vector $v$ is an eigenvector of every element of $G(\Omega_3, v|w|^2)$, from which it follows that $G(\Omega_3, v|w|^2)$ does not act transitively on $\Omega_3$. Therefore, assume that $v \in \partial \Omega_3 \setminus \left\{ 0 \right\}$ and recall from the analysis of the $k=3, n=4$ case in the previous section that in this situation we have $\dim \mathfrak g_0=4.$ In addition (see \cite[Lemma 3.8 and Proposition A.3]{Isa1}), if $v \in \partial \Omega_3 \setminus \left\{ 0 \right\}$ we have $\dim \mathfrak g_{1/2}=0$ and $\dim \mathfrak g_{1}=1$. 
So we see 
\[ d(D_4) = \dim \mathfrak g_{-1} + \dim \mathfrak g_{-1/2} + \dim \mathfrak g_0 + \dim \mathfrak g_1 = 10. \] 
Since $d(D_4) = 10 > 8 = n^2 - 8,$ we see that $S(\Omega, H)$ is not equivalent to $D_4$, and so Case 2 contributes nothing to our classification. 
\vspace{2mm} 

\textbf{Case 3.} 
Suppose that $k = 3, n = 5$. Here, $S(\Omega, H)$ is linearly equivalent either to 
$$D_5 := \left\{ (z, w) \in \mathbb C^3 \times \mathbb C^2: \text{Im } z - \mathcal H(w,w) \in \Omega_2 \right\}, $$ 
where $\mathcal H $ is an $\Omega_2$-Hermitian form, or to 
$$D_6 := \left\{ (z, w) \in \mathbb C^3 \times \mathbb C^2: \text{Im } z - \mathcal H(w,w) \in \Omega_3 \right\}, $$ 
where $\mathcal H$ is an $\Omega_3$-Hermitian form. Consideration of the domain $D_5$ does not aid our classification since $S(\Omega, H)$ must be biholomorphic to a five-dimensional product of three unit balls, and the only possibilities are $B^1 \times B^1 \times B^3$ and $B^1 \times B^2 \times B^2$. Since neither has automorphism group dimension $17 = n^2 - 8$, we assume then that $S(\Omega, H)$ is equivalent to the domain $D_6$.

By Lemma \ref{dimsubspaceskewherm} we have either $s = 1$, $s = 2$ or $s = 4$. In \cite{Isa3}, each of these scenarios was dealt with in Sections 5, 4 and 3 respectively. When $s=4$ we have $d(D_6) = 15 < 17 = n^2-8$, and when $s=2$ the action of $G(\Omega_3, \mathcal H)$ on $\Omega_3$ is not transitive. So consider the situation when $s=1$. In \cite[Lemma 5.1]{Isa3} it was shown that for the domain $D_6$ with $s=1$ and $\mathfrak g = \mathfrak g(D_6)$ we have $\dim \mathfrak g_{1/2} \le 2.$ We now prove a stronger result. 

\begin{lemma} \label{finallemma} 
For the domain $D_6$ with $s=1$ and $\mathfrak{g}=\mathfrak{g}(D_6)$ we have $\dim\mathfrak{g}_{1/2}= 0$.
\end{lemma}

\begin{proof} Let us write the $\Omega_3$-Hermitian form ${\mathcal H}$ as
$$
{\mathcal H}=u|w_1|^2+v|w_2|^2+a\bar w_1w_2+\bar a\, \bar w_2 w_1,
$$
where $u,v\in\mathbb R^3$ and $a\in\mathbb C^3$. Choosing $w_1=0$ and $w_2=0$ shows that $u,v\in\bar\Omega_3\setminus\{0\}$. We will consider two cases.

{\bf Case (i).} Suppose first that $u\in\Omega_3$. Then, as the cone $\Omega_3$ is homogeneous, we may assume that $u=(1,0,0)$. Further, replacing $w_1$ by $w_1+a_1w_2$, we may suppose that $a_1=0$. The above steps allow us to reduce $\mathcal H$ to the form 
\[ \mathcal H = \left[ \begin{array}{c} |w_1|^2 +v_1|w_2|^2 \\ v_2|w_2|^2 +a_2\, \bar w_1w_2+\bar a_2\, \bar w_2 w_1 \\ v_3|w_2|^2+ a_3\, \bar w_1w_2+ \bar a_3\, \bar w_2 w_1 \end{array} \right]. \]

\begin{remark} 
In \cite{Isa3}, Isaev further reduced the Hermitian form above by rotating the variables $z_2, z_3$ by a transformation from $\operatorname{O}_2,$ and thus assumed that $\mathcal H_3$ has no $|w_2|^2$-term, that is, $v_3=0$. We refrain from taking this step and assume the variable $v_3$ is not necessarily zero. 
\end{remark} 

To utilise Theorem \ref{descrg1/2}, let $\Phi:\mathbb C^3\to\mathbb C^2$ be a $\mathbb C$-linear map
\begin{equation}
\Phi(z_1,z_2,z_3)=\left(\varphi_1^1z_1+\varphi_2^1z_2+\varphi_3^1z_3, \varphi_1^2z_1+\varphi_2^2z_2+\varphi_3^2z_3\right),\label{mapPhineww}
\end{equation}
where $\varphi_i^j\in\mathbb C$. Fixing ${\mathbf w}\in\mathbb C^2$, for $x\in\mathbb R^3$ we compute 
\begin{equation} 
\begin{split} 
\mathcal H (\mathbf w, \Phi(x)) &= \big( \bar{\mathbf w}_1(\varphi_1^1x_1+\varphi_2^1x_2+\varphi_3^1x_3)+
v_1\bar{\mathbf w}_2(\varphi_1^2x_1+\varphi_2^2x_2+\varphi_3^2x_3), \\[7pt] 
& \hspace{1.3cm} v_2\bar{\mathbf w}_2(\varphi_1^2x_1+\varphi_2^2x_2+\varphi_3^2x_3)+a_2\bar{\mathbf w}_1(\varphi_1^2x_1+\varphi_2^2x_2+\varphi_3^2x_3) \\[7pt] 
& \hspace{0.8cm} + \bar{a}_2\bar{\mathbf w}_2(\varphi_1^1x_1+\varphi_2^1x_2+\varphi_3^1x_3), v_3 \bar{\mathbf w}_2(\varphi_1^2x_1+\varphi_2^2x_2+\varphi_3^2x_3) \\[7pt] 
& \hspace{0.8cm} + a_3 \bar{\mathbf w}_1(\varphi_1^2x_1+\varphi_2^2x_2+\varphi_3^2x_3) + \bar a_3 \bar{\mathbf w}_2(\varphi_1^1x_1+\varphi_2^1x_2+\varphi_3^1x_3) \big) \\[7pt]
&= \big( (\varphi_1^1\bar{\mathbf w}_1+v_1\varphi_1^2\bar{\mathbf w}_2)x_1+(\varphi_2^1\bar{\mathbf w}_1+v_1\varphi_2^2\bar{\mathbf w}_2)x_2 \\[7pt]
& \hspace{0.8cm} +(\varphi_3^1\bar{\mathbf w}_1+v_1\varphi_3^2\bar{\mathbf w}_2)x_3, (a_2\varphi_1^2\bar{\mathbf w}_1+(\bar a_2\varphi_1^1+v_2\varphi_1^2)\bar{\mathbf w}_2)x_1 \\[7pt] 
& \hspace{0.8cm} + (a_2\varphi_2^2\bar{\mathbf w}_1+(\bar a_2\varphi_2^1+v_2\varphi_2^2)\bar{\mathbf w}_2)x_2 + (a_2\varphi_3^2\bar{\mathbf w}_1 \\[7pt] 
& \hspace{0.8cm} + (\bar a_2\varphi_3^1+v_2\varphi_3^2)\bar{\mathbf w}_2)x_3, (a_3 \varphi_1^2\bar{\mathbf w}_1+(\bar a_3 \varphi_1^1+v_3 \varphi_1^2)\bar{\mathbf w}_2)x_1 \\[7pt] 
& \hspace{0.8cm} + (a_3 \varphi_2^2\bar{\mathbf w}_1+(\bar a_3\varphi_2^1+v_3 \varphi_2^2)\bar{\mathbf w}_2)x_2+(a_3 \varphi_3^2\bar{\mathbf w}_1+(\bar a_3 \varphi_3^1+v_3 \varphi_3^2)\bar{\mathbf w}_2)x_3 \big). 
\nonumber 
\end{split} 
\end{equation} 
Then from formula (\ref{Phiw0}) we see
\begin{equation} 
\begin{split} 
\Phi_{{\mathbf w}}(x)&= \big( (\operatorname{Im}(\varphi_1^1\bar{\mathbf w}_1)+v_1\operatorname{Im}(\varphi_1^2\bar{\mathbf w}_2))x_1+(\operatorname{Im}(\varphi_2^1\bar{\mathbf w}_1)+v_1\operatorname{Im}(\varphi_2^2\bar{\mathbf w}_2))x_2 \\[7pt] 
& \hspace{0.8cm} + (\operatorname{Im}(\varphi_3^1\bar{\mathbf w}_1)+v_1\operatorname{Im}(\varphi_3^2\bar{\mathbf w}_2))x_3, (\operatorname{Im}(a_2\varphi_1^2\bar{\mathbf w}_1) \\[7pt] 
& \hspace{0.8cm} + \operatorname{Im}((\bar a_2\varphi_1^1+v_2\varphi_1^2)\bar{\mathbf w}_2))x_1 + (\operatorname{Im}(a_2\varphi_2^2\bar{\mathbf w}_1) \\[7pt] 
& \hspace{0.8cm} + \operatorname{Im}((\bar a_2\varphi_2^1+v_2\varphi_2^2)\bar{\mathbf w}_2))x_2 + (\operatorname{Im}(a_2\varphi_3^2\bar{\mathbf w}_1) \\[7pt] 
& \hspace{0.8cm} + \operatorname{Im}((\bar a_2\varphi_3^1 + v_2\varphi_3^2)\bar{\mathbf w}_2))x_3, (\operatorname{Im}(a_3 \varphi_1^2\bar{\mathbf w}_1)+\operatorname{Im}((\bar a_3 \varphi_1^1+v_3 \varphi_1^2)\bar{\mathbf w}_2))x_1 \\[7pt] 
& \hspace{0.8cm} + (\operatorname{Im}(a_3 \varphi_2^2\bar{\mathbf w}_1) + \operatorname{Im}((\bar a_3 \varphi_2^1+ v_3 \varphi_2^2)\bar{\mathbf w}_2))x_2+(\operatorname{Im}(a_3 \varphi_3^2\bar{\mathbf w}_1) \\[7pt] 
& \hspace{1.8cm} + \operatorname{Im}((\bar a_3 \varphi_3^1+v_3 \varphi_3^2)\bar{\mathbf w}_2))x_3 \big). 
\nonumber 
\end{split} 
\end{equation}

Using (\ref{lieomegathree}), we then see that the condition that $\Phi_{{\mathbf w}}$ lies in ${\mathfrak g}(\Omega_3)$ for every ${\mathbf w}\in\mathbb C^2$ leads to the relations
\begin{align} 
\varphi_1^1&=a_2\varphi_2^2=a_3 \varphi_3^2 \nonumber \\[3pt] 
v_1\varphi_1^2&=\bar a_2\varphi_2^1+v_2\varphi_2^2=\bar a_3 \varphi_3^1+v_3 \varphi_3^2 \nonumber \\[3pt] 
\varphi_2^1&=a_2\varphi_1^2 \nonumber \\[3pt] 
v_1\varphi_2^2&=\bar a_2\varphi_1^1+v_2\varphi_1^2 \nonumber \\[3pt] 
\varphi_3^1&=a_3 \varphi_1^2 \nonumber \\[3pt] 
\bar a_3 \varphi_1^1+v_3 \varphi_1^2&=v_1\varphi_3^2 \nonumber \\[3pt] 
a_3 \varphi_2^2&=-a_2\varphi_3^2 \nonumber \\[3pt] 
\bar a_3 \varphi_2^1+v_3 \varphi_2^2&=-\bar a_2\varphi_3^1-v_2\varphi_3^2. \nonumber 
\end{align} 

If $a_2=0$, it immediately follows that $\Phi=0$. Similarly, if $a_3=0$, it also immediately follows that $\Phi=0.$ If both $a_2=0$ and $a_3=0$, a short row echelon computation shows that $\Phi=0.$ Thus by formula (\ref{cond1}) we have $\mathfrak{g}_{1/2}=0$. Suppose then that $a_2\ne 0$ and $a_3 \ne 0$. By scaling $w_2$, we can assume $a_3=1$. Then it follows that all $\varphi_i^j = 0$ unless $v_1=1, v_2=0, v_3=0$ and $a_2 = \pm i.$ We provide a brief sketch of the argument used to show this, which amounts to a standard row reduction of a large matrix. Writing the ten equations given above in matrix form, we have 
\[ 
\left[ \begin{array}{cccccc} 
1&0&0&0&-a_2&0 \\ 
1&0&0&0&0&-1 \\ 
0&-\bar{a}_2&0&v_1&-v_2&0 \\ 
0&0&-1&v_1&0&-v_3 \\ 
0&1&0&-a_2&0&0 \\ 
-\bar{a}_2&0&0&-v_2&v_1&0 \\ 
0&0&1&-1&0&0 \\ 
-1&0&0&-v_3&0&v_1 \\ 
0&0&0&0&1&a_2 \\ 
0&1&\bar{a}_2&0&v_3&v_2 
\end{array} \right] 
\left[ \begin{array}{c} \varphi_1^1 \\ \varphi_2^1 \\ \varphi_3^1 \\ \varphi_1^2 \\ \varphi_2^2 \\ \varphi_3^2 \end{array} \right] = 
\left[ \begin{array}{c} 0 \\ 0 \\ 0 \\ 0 \\ 0 \\ 0 \\ 0 \\ 0 \\ 0 \\ 0 \end{array} \right].
\]

We proceed to row reduce this matrix. After securing pivots in the first three columns, we can focus on the remaining $7 \times 3$ matrix. We begin the reduction of this matrix by assuming $v_1 \ne 1$. After securing pivots in the first two columns, by then varying the values of $a_2$ we can always get a pivot in the third column. We then assume $v_1=1$. We see at this stage that if $v_3$ is non-zero, we get a pivot in every column, and therefore assume $v_3 = 0$. Continuing in this fashion, we eventually see that only when assuming $v_1=1, v_2=0, v_3=0$ and $a_2 = \pm i$ does the matrix fail to be full rank. Therefore, in situations other than this we have $\varphi_i^j = 0$ for all $i, j$. Then $\Phi = 0$, and by formula (\ref{cond1}) we have $\mathfrak{g}_{1/2}=0$.

We thus assume that $v_1=1, v_2=0, v_3=0$ and $a_2 = \pm i$. In this situation, the Hermitian form $\mathcal H$ is given by 
$$
{\mathcal H}=(|w_1|^2+|w_2|^2,\pm i(\bar w_1w_2-\bar w_2 w_1), \bar w_1w_2+\bar w_2 w_1).
$$
Changing the $w$-variables as
$$
w_1\mapsto -\frac{i}{\sqrt{2}}(w_1+iw_2),\,\,w_2\mapsto \frac{1}{\sqrt{2}}(w_1-iw_2),
$$
we can suppose that
$$
{\mathcal H}=(|w_1|^2+|w_2|^2,\mp(|w_1|^2-|w_2|^2), \bar w_1w_2+\bar w_2 w_1).
$$
Further, swapping $w_1$ and $w_2$ if necessary, we reduce our considerations to the case where
\begin{equation}
{\mathcal H}=(|w_1|^2+|w_2|^2,|w_1|^2-|w_2|^2, \bar w_1w_2+\bar w_2 w_1).\label{verynewformcalh}
\end{equation}

We will now show that for the above $\Omega_3$-Hermitian form ${\mathcal H}$ one has $\mathfrak{g}_{1/2}=0$. Consider a map $\Phi:\mathbb C^3\to\mathbb C^2$ as in (\ref{mapPhineww}), fix ${\mathbf w}\in\mathbb C^2$, and for $x\in\mathbb R^3$ compute
\begin{equation} 
\begin{split} 
\mathcal H (\mathbf w, \Phi(x)) &= \big( \bar{\mathbf w}_1 (\varphi_1^1 x_1 + \varphi_2^1 x_2 + \varphi_3^1 x_3) + \bar{\mathbf w}_2 (\varphi_1^2 x_1 + \varphi_2^2 x_2 + \varphi_3^2 x_3), \\[7pt] 
& \hspace{0.8cm} \bar{\mathbf w}_1 (\varphi_1^1 x_1 + \varphi_2^1 x_2 + \varphi_3^1 x_3) - \bar{\mathbf w}_2 (\varphi_1^2 x_1 + \varphi_2^2 x_2 + \varphi_3^2 x_3), \\[7pt] 
& \hspace{0.8cm} \bar{\mathbf w}_1 (\varphi_1^2 x_1 + \varphi_2^2 x_2 + \varphi_3^2 x_3) + \bar{\mathbf w}_2 (\varphi_1^1 x_1 + \varphi_2^1 x_2 + \varphi_3^1 x_3) \big) \\[7pt] 
&= \big( (\varphi_1^1\bar{\mathbf w}_1+\varphi_1^2\bar{\mathbf w}_2)x_1+(\varphi_2^1\bar{\mathbf w}_1+\varphi_2^2\bar{\mathbf w}_2)x_2+(\varphi_3^1\bar{\mathbf w}_1+\varphi_3^2\bar{\mathbf w}_2)x_3, \\[7pt] 
& \hspace{0.8cm} (\varphi_1^1\bar{\mathbf w}_1-\varphi_1^2\bar{\mathbf w}_2)x_1+(\varphi_2^1\bar{\mathbf w}_1-\varphi_2^2\bar{\mathbf w}_2)x_2+(\varphi_3^1\bar{\mathbf w}_1-\varphi_3^2\bar{\mathbf w}_2)x_3, \\[7pt] 
& \hspace{0.8cm} (\varphi_1^2\bar{\mathbf w}_1+\varphi_1^1\bar{\mathbf w}_2)x_1+(\varphi_2^2\bar{\mathbf w}_1+\varphi_2^1\bar{\mathbf w}_2)x_2+(\varphi_3^2\bar{\mathbf w}_1+\varphi_3^1\bar{\mathbf w}_2)x_3 \big). 
\nonumber 
\end{split} 
\end{equation} 
Then from formula (\ref{Phiw0}) we see
$$
\begin{array}{l}
\displaystyle\Phi_{{\mathbf w}}(x)=\left((\operatorname{Im}(\varphi_1^1\bar{\mathbf w}_1)+\operatorname{Im}(\varphi_1^2\bar{\mathbf w}_2))x_1+(\operatorname{Im}(\varphi_2^1\bar{\mathbf w}_1)+\operatorname{Im}(\varphi_2^2\bar{\mathbf w}_2))x_2\right.\\
\vspace{-0.1cm}\\
\hspace{1.8cm}\left.+(\operatorname{Im}(\varphi_3^1\bar{\mathbf w}_1)+\operatorname{Im}(\varphi_3^2\bar{\mathbf w}_2))x_3,(\operatorname{Im}(\varphi_1^1\bar{\mathbf w}_1)-\operatorname{Im}(\varphi_1^2\bar{\mathbf w}_2))x_1\right.\\
\vspace{-0.1cm}\\
\hspace{1.8cm}\left.+(\operatorname{Im}(\varphi_2^1\bar{\mathbf w}_1)-\operatorname{Im}(\varphi_2^2\bar{\mathbf w}_2))x_2+(\operatorname{Im}(\varphi_3^1\bar{\mathbf w}_1)-\operatorname{Im}(\varphi_3^2\bar{\mathbf w}_2))x_3,\right.\\
\vspace{-0.1cm}\\
\hspace{2.1cm}\left.(\operatorname{Im}(\varphi_1^2\bar{\mathbf w}_1)+\operatorname{Im}(\varphi_1^1\bar{\mathbf w}_2))x_1+(\operatorname{Im}(\varphi_2^2\bar{\mathbf w}_1)+\operatorname{Im}(\varphi_2^1\bar{\mathbf w}_2))x_2\right.\\
\vspace{-0.1cm}\\
\hspace{3.0cm}\left.+(\operatorname{Im}(\varphi_3^2\bar{\mathbf w}_1)+\operatorname{Im}(\varphi_3^1\bar{\mathbf w}_2))x_3\right).
\end{array}
$$
From (\ref{lieomegathree}) we then see that the condition that $\Phi_{{\mathbf w}}$ lies in ${\mathfrak g}(\Omega_3)$ for every ${\mathbf w}\in\mathbb C^2$ leads to the relations
\begin{equation}
\begin{array}{l}
\varphi_1^1=\varphi_2^1=\varphi_3^2,\,\,\varphi_1^2=-\varphi_2^2=\varphi_3^1,\\
\end{array}\label{relmapphibig11}
\end{equation}

Further, let $c$ be a symmetric $\mathbb C$-bilinear form on $\mathbb C^2$ with values in $\mathbb C^2$:
$$
c(w,w)=\left(c^1_{11}w_1^2+2c^1_{12}w_1w_2+c^1_{22}w_2^2,c^2_{11}w_1^2+2c^2_{12}w_1w_2+c^2_{22}w_2^2\right),
$$
where $c^{\ell}_{ij}\in\mathbb C$. Then for $w,w'\in\mathbb C^2$ using (\ref{verynewformcalh}) we calculate
\begin{equation} 
\begin{split} 
{\mathcal H}(w,c(w',w'))&=\big( \bar w_1(c^1_{11}(w_1')^{2}+2c^1_{12}w_1'w_2'+c^1_{22}(w_2')^{2})+\bar w_2(c^2_{11}(w_1')^{2} \\[7pt] 
& \hspace{0.8cm} + 2c^2_{12}w_1'w_2'+c^2_{22}(w_2')^{2}), \bar w_1(c^1_{11}(w_1')^{2}+2c^1_{12}w_1'w_2' \\[7pt] 
& \hspace{0.8cm} + c^1_{22}(w_2')^{2}) -w_2(c^2_{11}(w_1')^{2}+2c^2_{12}w_1'w_2'+c^2_{22}(w_2')^{2}), \\[7pt] 
& \hspace{1.3cm} \bar w_1(c^2_{11}(w_1')^{2}+2c^2_{12}w_1'w_2' + c^2_{22}(w_2')^{2}) \\[7pt] 
& \hspace{0.8cm} + \bar w_2(c^1_{11}(w_1')^{2}+2c^1_{12}w_1'w_2'+c^1_{22}(w_2')^{2}) \big). 
\end{split} \label{expressdiff1}
\end{equation} 
On the other hand, we have
\begin{equation} 
\begin{split} 
\Phi({\mathcal H}(w',w))&=\big(\varphi_1^1(\bar w_1'w_1+\bar w_2'w_2)+\varphi_2^1(\bar w_1'w_1-\bar w_2'w_2)+\varphi_3^1(\bar w_1' w_2+\bar w_2'w_1), \\[7pt] 
& \hspace{0.8cm} \varphi_1^2(\bar w_1'w_1+\bar w_2'w_2)+\varphi_2^2(\bar w_1'w_1-\bar w_2'w_2)+\varphi_3^2(\bar w_1' w_2+\bar w_2'w_1)\big) \\[7pt] 
& = \big( (\varphi_1^1+\varphi_2^1)\bar w_1'w_1+(\varphi_1^1-\varphi_2^1)\bar w_2'w_2+\varphi_3^1(\bar w_1' w_2+\bar w_2'w_1), \\[7pt] 
& \hspace{0.8cm} (\varphi_1^2+\varphi_2^2)\bar w_1'w_1+(\varphi_1^2-\varphi_2^2)\bar w_2'w_2+\varphi_3^2(\bar w_1' w_2+\bar w_2'w_1)\big). 
\nonumber 
\end{split} 
\end{equation} 
Therefore
\begin{equation}
\makebox[250pt]{$\begin{array}{l}
2i{\mathcal H}(\Phi({\mathcal H}(w',w)),w')=\\
\vspace{-0.1cm}\\
\hspace{0.8cm}
2i\left(w_1'\left((\bar\varphi_1^1+\bar\varphi_2^1)w_1'\bar w_1+(\bar\varphi_1^1-\bar\varphi_2^1) w_2'\bar w_2+\bar\varphi_3^1(w_1' \bar w_2+w_2'\bar w_1)\right)\right.\\
\vspace{-0.1cm}\\
\hspace{1.2cm}\left.+w_2'\left((\bar\varphi_1^2+\bar\varphi_2^2)w_1'\bar w_1+(\bar\varphi_1^2-\bar\varphi_2^2) w_2'\bar w_2+\bar\varphi_3^2(w_1' \bar w_2+w_2'\bar w_1)\right),\right.\\
\vspace{-0.1cm}\\
\hspace{1.5cm}\left.w_1'\left((\bar\varphi_1^1+\bar\varphi_2^1)w_1'\bar w_1+(\bar\varphi_1^1-\bar\varphi_2^1) w_2'\bar w_2+\bar\varphi_3^1(w_1' \bar w_2+w_2'\bar w_1)\right)\right.\\
\vspace{-0.1cm}\\
\hspace{1.2cm}\left.-w_2'\left((\bar\varphi_1^2+\bar\varphi_2^2)w_1'\bar w_1+(\bar\varphi_1^2-\bar\varphi_2^2) w_2'\bar w_2+\bar\varphi_3^2(w_1' \bar w_2+w_2'\bar w_1)\right),\right.\\
\vspace{-0.1cm}\\
\hspace{1.5cm}\left. w_1'\left((\bar\varphi_1^2+\bar\varphi_2^2)w_1'\bar w_1+(\bar\varphi_1^2-\bar\varphi_2^2) w_2'\bar w_2+\bar\varphi_3^2(w_1' \bar w_2+w_2'\bar w_1)\right)\right.\\
\vspace{-0.1cm}\\
\hspace{1.2cm}\left.+w_2'\left((\bar\varphi_1^1+\bar\varphi_2^1)w_1'\bar w_1+(\bar\varphi_1^1-\bar\varphi_2^1) w_2'\bar w_2+\bar\varphi_3^1(w_1' \bar w_2+w_2'\bar w_1)\right)\right).
\end{array}$}\label{expressdiff2}
\end{equation}
Let us now compare expressions (\ref{expressdiff1}) and (\ref{expressdiff2}) as required by condition (\ref{cond1}). Specifically, looking at the coefficients of $(w_2')^2\bar w_1$ and $(w_1')^2\bar w_2$ in the first and second components of these expressions, we obtain the identities:
$$
c_{22}^1=2i\bar\varphi_3^2,\qquad c_{22}^1=-2i\bar\varphi_3^2,\qquad c_{11}^2=2i\bar\varphi_3^1,\qquad -c_{11}^2=2i\bar\varphi_3^1,
$$
which imply $\varphi_3^1=0$, $\varphi_3^2=0$. Taken together with (\ref{relmapphibig11}), these conditions yield $\Phi=0$, hence ${\mathfrak g}_{1/2}=0$ as required.

{\bf Case (ii).} Suppose now that $u\in\partial\Omega_3\setminus\{0\}$. In this situation, as the group $G(\Omega_3)^{\circ}=\mathbb R_{+}\times \operatorname{SO}_{1,2}^{\circ}$ acts transitively on $\partial\Omega_3\setminus\{0\}$, we may assume that\linebreak $u=(1,1,0)$. Further, replacing $w_1$ by $w_1+a_1w_2$, we may suppose that $a_1=0$. The above steps allow us to reduce $\mathcal H$ to the form 
\[ \mathcal H = \left[ \begin{array}{c} |w_1|^2 +v_1|w_2|^2 \\ |w_1|^2 + v_2|w_2|^2 +a_2\bar w_1w_2+\bar a_2\, \bar w_2 w_1 \\ v_3|w_2|^2 +a_3\bar w_1w_2+\bar a_3\, \bar w_2 w_1 \end{array} \right]. \]

Let $\Phi:\mathbb C^3\to\mathbb C^2$ be a $\mathbb C$-linear map as in (\ref{mapPhineww}). Fixing ${\mathbf w}\in\mathbb C^2$, for $x\in\mathbb R^3$ we compute
\begin{equation} 
\begin{split} 
\mathcal H (\mathbf w, \Phi(x)) &= \big( \bar{\mathbf w}_1(\varphi_1^1x_1+\varphi_2^1x_2+\varphi_3^1x_3)+
v_1\bar{\mathbf w}_2(\varphi_1^2x_1+\varphi_2^2x_2+\varphi_3^2x_3), \\[7pt] 
& \hspace{1.3cm} \bar{\mathbf w}_1(\varphi_1^1x_1+\varphi_2^1x_2+\varphi_3^1x_3)+v_2\bar{\mathbf w}_2(\varphi_1^2x_1+\varphi_2^2x_2+\varphi_3^2x_3) \\[7pt] 
& \hspace{0.8cm} + a_2\bar{\mathbf w}_1(\varphi_1^2x_1+\varphi_2^2x_2+\varphi_3^2x_3)+\bar a_2\bar{\mathbf w}_2(\varphi_1^1x_1+\varphi_2^1x_2+\varphi_3^1x_3), \\[7pt] 
& \hspace{1.3cm} (\varphi_1^2x_1+\varphi_2^2x_2+\varphi_3^2x_3)+a_3\bar{\mathbf w}_1(\varphi_1^2x_1+\varphi_2^2x_2+\varphi_3^2x_3) \\[7pt] 
& \hspace{0.8cm} + \bar a_3\bar{\mathbf w}_2(\varphi_1^1x_1+\varphi_2^1x_2+\varphi_3^1x_3) \big) \\[7pt] 
&= \big( (\varphi_1^1\bar{\mathbf w}_1+v_1\varphi_1^2\bar{\mathbf w}_2)x_1+ (\varphi_2^1\bar{\mathbf w}_1+v_1\varphi_2^2\bar{\mathbf w}_2)x_2 \\[7pt] 
& \hspace{0.8cm} + (\varphi_3^1\bar{\mathbf w}_1+v_1\varphi_3^2\bar{\mathbf w}_2)x_3, ( (\varphi_1^1+a_2\varphi_1^2)\bar{\mathbf w}_1+ (\bar a_2\varphi_1^1+v_2\varphi_1^2)\bar{\mathbf w}_2)x_1 \\[7pt] 
& \hspace{0.8cm} + ((\varphi_2^1+a_2\varphi_2^2)\bar{\mathbf w}_1+(\bar a_2\varphi_2^1+v_2\varphi_2^2)\bar{\mathbf w}_2)x_2 +((\varphi_3^1+a_2\varphi_3^2)\bar{\mathbf w}_1 \\[7pt] 
& \hspace{0.8cm} +(\bar a_2\varphi_3^1+v_2\varphi_3^2)\bar{\mathbf w}_2)x_3, (a_3\varphi_1^2\bar{\mathbf w}_1+(\bar a_3\varphi_1^1+v_3\varphi_1^2)\bar{\mathbf w}_2)x_1 \\[7pt] 
& \hspace{0.8cm} + (a_3\varphi_2^2\bar{\mathbf w}_1+(\bar a_3\varphi_2^1+v_3\varphi_2^2)\bar{\mathbf w}_2)x_2+(a_3\varphi_3^2\bar{\mathbf w}_1 \\[7pt] 
& \hspace{1.8cm} +(\bar a_3\varphi_3^1+v_3\varphi_3^2)\bar{\mathbf w}_2)x_3 \big). 
\nonumber 
\end{split} 
\end{equation} 
Then from formula (\ref{Phiw0}) we see
\begin{equation} 
\begin{split} 
\Phi_{{\mathbf w}}(x) &= \big((\operatorname{Im}(\varphi_1^1\bar{\mathbf w}_1)+v_1\operatorname{Im}(\varphi_1^2\bar{\mathbf w}_2))x_1+(\operatorname{Im}(\varphi_2^1\bar{\mathbf w}_1)+v_1\operatorname{Im}(\varphi_2^2\bar{\mathbf w}_2))x_2 \\[7pt] 
& \hspace{0.8cm} + (\operatorname{Im}(\varphi_3^1\bar{\mathbf w}_1)+v_1\operatorname{Im}(\varphi_3^2\bar{\mathbf w}_2))x_3, ( \operatorname{Im}((\varphi_1^1+a_2\varphi_1^2)\bar{\mathbf w}_1) \\[7pt] 
& \hspace{0.8cm} + \operatorname{Im}((\bar a_2\varphi_1^1+v_2\varphi_1^2)\bar{\mathbf w}_2))x_1+ (\operatorname{Im}((\varphi_2^1+a_2\varphi_2^2)\bar{\mathbf w}_1) \\[7pt] 
& \hspace{0.8cm} +\operatorname{Im}((\bar a_2\varphi_2^1+v_2\varphi_2^2)\bar{\mathbf w}_2))x_2 +(\operatorname{Im}((\varphi_3^1+a_2\varphi_3^2)\bar{\mathbf w}_1) \\[7pt] 
& \hspace{0.8cm} + \operatorname{Im}((\bar a_2\varphi_3^1+v_2\varphi_3^2)\bar{\mathbf w}_2))x_3, (\operatorname{Im}(a_3\varphi_1^2\bar{\mathbf w}_1)+\operatorname{Im}((\bar a_3\varphi_1^1+v_3\varphi_1^2)\bar{\mathbf w}_2))x_1 \\[7pt] 
& \hspace{0.8cm} + (\operatorname{Im}(a_3\varphi_2^2\bar{\mathbf w}_1)+\operatorname{Im}((\bar a_3\varphi_2^1+v_3\varphi_2^2)\bar{\mathbf w}_2))x_2+(\operatorname{Im}(a_3\varphi_3^2\bar{\mathbf w}_1) \\[7pt] 
& \hspace{1.8cm} + \operatorname{Im}((\bar a_3\varphi_3^1+v_3\varphi_3^2)\bar{\mathbf w}_2))x_3 \big). 
\nonumber 
\end{split} 
\end{equation}

Using (\ref{lieomegathree}), we then see that the condition that $\Phi_{{\mathbf w}}$ lies in ${\mathfrak g}(\Omega_3)$ for every ${\mathbf w}\in\mathbb C^2$ leads to the relations
\begin{align} 
\varphi_1^1&=\varphi_2^1+a_2\varphi_2^2=a_3\varphi_3^2 \nonumber \\[3pt] 
v_1\varphi_1^2&=\bar a_2\varphi_2^1+v_2\varphi_2^2=\bar a_3\varphi_3^1+v_3\varphi_3^2 \nonumber \\[3pt] 
\varphi_2^1&=\varphi_1^1+a_2\varphi_1^2 \nonumber \\[3pt] 
v_1\varphi_2^2&=\bar a_2\varphi_1^1+v_2\varphi_1^2 \nonumber \\[3pt] 
\varphi_3^1&=a_3\varphi_1^2 \nonumber \\[3pt] 
\bar a_3\varphi_1^1+v_3\varphi_1^2&=v_1\varphi_3^2 \nonumber \\[3pt] 
a_3\varphi_2^2&=-\varphi_3^1-a_2\varphi_3^2 \nonumber \\[3pt] 
\bar a_3\varphi_2^1+v_3\varphi_2^2&=-\bar a_2\varphi_3^1-v_2\varphi_3^2. \nonumber 
\end{align} 

It easily follows that if $a_3=0$, then $\Phi=0$, so by formula (\ref{cond1}) we have ${\mathfrak g}_{1/2}=0$. If $a_3\ne 0$, then, by scaling $w_2$, we can assume that $a_3=1$. Similarly to case $(i)$, we consider the above ten equations in matrix form and row reduce. We have 
\[ 
\left[ \begin{array}{cccccc} 
1&-1&0&0&-a_2&0 \\ 
1&0&0&0&0&-1 \\ 
0&-\bar{a}_2&0&v_1&-v_2&0 \\ 
0&0&-1&v_1&0&-v_3 \\ 
-1&1&0&-a_2&0&0 \\ 
-\bar{a}_2&0&0&-v_2&v_1&0 \\ 
0&0&1&-1&0&0 \\ 
1&0&0&v_3&0&-v_1 \\ 
0&0&1&0&1&a_2 \\ 
0&1&\bar{a}_2&0&v_3&v_2 
\end{array} \right] 
\left[ \begin{array}{c} \varphi_1^1 \\ \varphi_2^1 \\ \varphi_3^1 \\ \varphi_1^2 \\ \varphi_2^2 \\ \varphi_3^2 \end{array} \right] = 
\left[ \begin{array}{c} 0 \\ 0 \\ 0 \\ 0 \\ 0 \\ 0 \\ 0 \\ 0 \\ 0 \\ 0 \end{array} \right].
\]

After easily securing pivots in the first four columns, we can focus on the resulting $6 \times 2$ matrix. We begin the row reduction of this matrix by assuming $a_2 \ne 0.$ Then by varying the values of $v_1, v_2$ and $v_3$ we see that this matrix is always full rank. Therefore, assume $a_2 = 0.$ By again varying the values of $v_1, v_2$ and $v_3$ we find the only situation in which the matrix is not full rank is when $a_2=0, v_1=1, v_2=-1,$ and  $v_3=0$. Therefore, in situations other than this we have $\varphi_i^j = 0$ for all $i, j$. Then $\Phi = 0$, and by formula (\ref{cond1}) we have $\mathfrak{g}_{1/2}=0$. Finally, notice that for the above values of $v_1$, $v_2$, $v_3$, $a_2$ the form ${\mathcal H}$ coincides with the right-hand side of (\ref{verynewformcalh}), for which we have already shown that ${\mathfrak g}_{1/2}=0$. 
\end{proof}

Now, Lemma \ref{finallemma} together with (\ref{estim 8}) and the second inequality in (\ref{estimm}) yields
$d(D_6)\le  15<17=n^2-8$. Thus, we have shown that Case (3) makes no contributions to the classification. 
\vspace{2mm} 

\textbf{Case 4.} 
Suppose that $k = 3, n = 6$. Here, $S(\Omega, H)$ is linearly equivalent either to 
$$D_7 := \left\{ (z, w) \in \mathbb C^3 \times \mathbb C^3: \text{Im } z - \mathcal H(w,w) \in \Omega_2 \right\}, $$ 
where $\mathcal H $ is an $\Omega_2$-Hermitian form, or to 
$$D_8 := \left\{ (z, w) \in \mathbb C^3 \times \mathbb C^3: \text{Im } z - \mathcal H(w,w) \in \Omega_3 \right\}, $$ 
where $\mathcal H$ is an $\Omega_3$-Hermitian form. 

Assume $S(\Omega, H)$ is equivalent to $D_7$. Then $S(\Omega, H)$ must be biholomorphic to a product of three unit balls. The only possibilities are $B^1 \times B^1 \times B^4$, $B^1 \times B^2 \times B^3$ or $B^2 \times B^2 \times B^2$, none of which have automorphism group of dimension $n^2 - 8 = 28$. 

So $S(\Omega, H)$ must be equivalent to $D_8$. By (\ref{estim2}) we have $s + \dim \mathfrak{g}(\Omega) \geq 10$. Since $\dim \mathfrak{g}(\Omega_3) = 4$, we see that $s \ge 6$. On the other hand, by (\ref{ests}) we have $s \leq 9$. Now recall from the corresponding case in the previous section that we must have $s = 1,2,3,5$ or $9$. Therefore we see that $s=9$ and the argument proceeds in the same way, showing that $S(\Omega, H)$ cannot be equivalent to $D_8$ and no contribution to the classification is made. 
\vspace{2mm} 

\textbf{Case 5.} 
Suppose that $k=3, n=7$. Here, $S(\Omega, H)$ is linearly equivalent either to 
$$D_9 := \left\{ (z, w) \in \mathbb C^3 \times \mathbb C^4: \text{Im } z - \mathcal H(w,w) \in \Omega_2 \right\}, $$ 
where $\mathcal H $ is an $\Omega_2$-Hermitian form, or to 
$$D_{10} := \left\{ (z, w) \in \mathbb C^3 \times \mathbb C^4: \text{Im } z - \mathcal H(w,w) \in \Omega_3 \right\}, $$ 
where $\mathcal H$ is an $\Omega_3$-Hermitian form. 
Inequality (\ref{estim2}) implies $s+ \dim \mathfrak g(\Omega) \ge 19$, and so we must consider the possibility that $s=16$ and $\Omega$ is linearly equivalent to $\Omega_2$. Then $S(\Omega, H)$ is equivalent to the domain $D_9$, which yields the product $B^1 \times B^1 \times B^5$, since $d(B^1 \times B^1 \times B^5) = 3+3+35 = 41 = n^2-8.$ It follows from the analysis of the $k=3,n=7$ case in the previous section that if $\Omega$ is linearly equivalent to $\Omega_3$, then Case $5$ provides no further contributions to the classification. 
\vspace{2mm} 

\textbf{Case 6.} 
Suppose that $k=4, n=4$. In this case, after a linear change of variables, $S(\Omega, H)$ is one of the domains 
$$\left\{ z \in \mathbb C^4 : \text{Im } z \in \Omega_4 \right\},$$ 
$$\left\{ z \in \mathbb C^4 : \text{Im } z \in \Omega_5 \right\},$$ 
$$\left\{ z \in \mathbb C^4 : \text{Im } z \in \Omega_6 \right\},$$ 
and therefore is biholomorphic either to $B^1 \times B^1 \times B^1 \times B^1$, or to $B^1 \times T_3$, or to $T_4$, as in the case of automorphism group dimension $d(M) = n^2-7.$ The dimensions of the respective automorphism groups of these domains are 12, 13 and 15. Each of these numbers is greater than $8 = n^2 - 8$, and so we see that Case 6 contributes nothing to our classification. 
\vspace{2mm} 

\textbf{Case 7.} 
Suppose that $k=4, n=5$. Then $S(\Omega, H)$ is linearly equivalent to either 
$$D_{11} := \left\{ (z, w) \in \mathbb C^4 \times \mathbb C: \text{Im } z - v|w|^2 \in \Omega_4 \right\}, $$ 
where $v = (v_1, v_2, v_3, v_4)$ is a vector in $\mathbb R^4$ with non-negative entries, or 
$$D_{12} := \left\{ (z, w) \in \mathbb C^4 \times \mathbb C: \text{Im } z - v|w|^2 \in \Omega_5 \right\}, $$ 
where $v = (v_1, v_2, v_3, v_4)$ is a vector in $\mathbb R^4$ satisfying $v \in \bar{\Omega}_5 \setminus \left\{ 0 \right\}$, or 
$$D_{13} := \left\{ (z, w) \in \mathbb C^4 \times \mathbb C: \text{Im } z - v|w|^2 \in \Omega_6 \right\}, $$ 
where $v = (v_1, v_2, v_3, v_4)$ is a vector in $\mathbb R^4$ satisfying $v \in \bar{\Omega}_6 \setminus \left\{ 0 \right\}$, i.e., $v_1^2 \geq v_2^2 + v_3^2 + v_4^2, v_1 > 0.$ 

Since $s=1$, by inequality (\ref{estim2}) we see that $\dim \mathfrak{g}(\Omega) \geq 4$. Therefore, in contrast to the corresponding case in the previous section, we must also consider the domain $D_{11}$. Let us begin with this possibility, and assume $S(\Omega, H)$ is equivalent to $D_{11}$. Then $S(\Omega, H)$ is biholomorphic to the product of unit balls given by $B^1 \times B^1 \times B^1 \times B^2$, since $d(B^1 \times B^1 \times B^1 \times B^2) = 3+3+3+8 = 17 = n^2-8$.

Next, assume that $S(\Omega, H)$ is equivalent to $D_{12}$. As in the analysis of the same case in the previous section, by Theorem (\ref{descrg1/2}) we see that in the cases of the boundary components $C_1, C_2$ and $C_3$, for $\mathfrak g = \mathfrak g(D_{12})$ we have $\mathfrak g_{1/2}=0$. Then by estimate (\ref{estim 8}), the second inequality in (\ref{estimm}) and Lemmas \ref{componentone}, \ref{componenttwo} and \ref{componentthree} we see that in each of these cases 
\[ d(D_{12}) \le 16 < 18 = n^2-7 \] 
(recall that $s=1$). This shows that in the cases of these components, $S(\Omega, H)$ cannot be equivalent to $D_{12}$, so no new contributions are made to our classification. 

Lastly, assume that $S(\Omega, H)$ is equivalent to $D_{13}$. Recall from the analysis of the same case in the previous section that for $\mathfrak g = \mathfrak g(D_{13})$ we have $\dim \mathfrak g_0 = 6$. Then using the second inequality in (\ref{estimm}) and \cite[Lemma 4.3]{Isa3} for $N = 1$, we see that 
$$d(D_{13}) = \dim \mathfrak{g}_{-1} + \dim \mathfrak{g}_{-1/2} + \dim \mathfrak{g}_{0} + \dim \mathfrak{g}_{1/2} + \dim \mathfrak{g}_{1} \leq 16 < 17 = n^2 - 8,$$ 
showing no contribution to the classification. Therefore, the product $B^1 \times B^1 \times B^1 \times B^2$ is the only contribution made to the classification by Case 7. 
\vspace{2mm} 

\textbf{Case 8.} 
Suppose that $k=5$ and $n=5$. Then by inequality (\ref{estim2}) we see that in this situation we have $\dim \mathfrak g(\Omega) \geq 7$. Therefore, after a linear change of variables $S(\Omega, H)$ turns into one of the domains 
\[ \left\{z \in \mathbb C^5: \operatorname{Im }z \in \Omega_9 \right\}, \]
\[ \left\{z \in \mathbb C^5: \operatorname{Im }z \in \Omega_{10} \right\}, \] 
and therefore is biholomorphic either to $B^1 \times T_4$, or to $T_5$, as in the case of automorphism group dimension $d(M) = n^2 - 7.$ The dimensions of the respective automorphism groups of these domains are $18$ and $21$. Each of these numbers is greater than $17 = n^2 - 8$, and so we see that Case $8$ makes no contribution to the classification. 
\vspace{2mm} 

\textbf{Case 9.} 
Suppose that $k=6, n=6$. Consider the following result from \cite[Lemma 2.1]{Isa3}, which we state without proof. 
\begin{lemma} 
If for $k\ge 3$ we set
\[ K:=\frac{(k-2)(k-3)}{2}+k+1, \]
then the inequality $\dim {\mathfrak g}(\Omega)\ge K$ implies that $\Omega$ is linearly equivalent to $\Lambda_k$. 
\end{lemma} 
In this case, inequality (\ref{estim2}) implies that $\dim \mathfrak g \ge 16 > 13 = K,$ and so by the above lemma $\Omega$ is linearly equivalent to the Lorentz cone $\Lambda_6.$ Therefore, after a linear change of variables, $S(\Omega, H)$ turns into the domain 
\[ \left\{ z \in \mathbb C^6 : \operatorname{Im} z \in \Lambda_6 \right\}, \] which is the tube domain $T_6$, where 
\begin{align}
T_6 = \big\{(z_1, z_2, z_3, z_4, z_5, z_6) & \in \mathbb C^6 : (\operatorname{Im } z_1)^2 -(\operatorname{Im } z_2)^2 - (\operatorname{Im } z_3)^2 \nonumber \\
&- (\operatorname{Im } z_4)^2 - (\operatorname{Im } z_5)^2 - (\operatorname{Im } z_6)^2 >0, \operatorname{Im }z_1 >0 \big\}. \nonumber 
\end{align} 
Note that $d(T_6) = 28 = n^2-8,$ so Case 9 contributes $T_6$ to the classification of homogeneous Kobayashi-hyperbolic manifolds with automorphism group dimension $n^2-8.$ This completes the proof. 

\bibliographystyle{plain}
\bibliography{ContinuingClassification}

\end{document}